\documentclass [twoside,reqno,  12pt] {amsart}

\usepackage{amsfonts}
\usepackage{amssymb}
\usepackage{a4}
\usepackage{color}

\newtheorem{thm}{Theorem}[section]
\newtheorem{cor}[thm]{Corollary}
\newtheorem{lem}[thm]{Lemma}
\newtheorem{prop}[thm]{Proposition}

\newtheorem{defn}[thm]{Definition}

\theoremstyle{definition}

\numberwithin{equation}{section}

\renewcommand{\Re}{\hbox{Re}\,}
\renewcommand{\Im}{\hbox{Im}\,}

\newcommand{\C}{\mathbb{C}}

\renewcommand{\div}{\operatorname{div}}

\newcommand{\loc}{\operatorname{loc}}

\newcommand{\R}{\mathbb{R}}

\newcommand{\supp}{\operatorname{supp}}

\def\hat{\widehat}
\def\tilde{\widetilde}
\def \bfo {\begin {eqnarray*} }
\def \efo {\end {eqnarray*} }
\def \ba {\begin {eqnarray*} }
\def \ea {\end {eqnarray*} }
\def \beq {\begin {eqnarray}}
\def \eeq {\end {eqnarray}}
\def \supp {\hbox{supp }}

\def \p {\partial}

\def\hat{\widehat}
\def\tilde{\widetilde}
\def \bfo {\begin {eqnarray*} }
\def \efo {\end {eqnarray*} }
\def \ba {\begin {eqnarray*} }
\def \ea {\end {eqnarray*} }
\def \beq {\begin {eqnarray}}
\def \eeq {\end {eqnarray}}
\def \supp {\hbox{supp }}

\def \p {\partial}

\begin{document}

 \title[Advection diffusion equations in admissible geometries]{Inverse problems for advection diffusion equations in admissible geometries}

\author[Krupchyk]{Katya Krupchyk}

\address
        {K. Krupchyk, Department of Mathematics\\
University of California, Irvine\\ 
CA 92697-3875, USA }

\email{katya.krupchyk@uci.edu}

\author[Uhlmann]{Gunther Uhlmann}

\address
       {G. Uhlmann, Department of Mathematics\\
       University of Washington\\
       Seattle, WA  98195-4350\\
       USA\\
       Department of Mathematics and Statistics\\ 
       University of Helsinki\\ 
       Finland\\
        and Institute for Advanced Study of the Hong Kong University of Science and Technology}
\email{gunther@math.washington.edu}

\maketitle

\begin{abstract} We study inverse boundary problems for the advection diffusion equation on an admissible manifold, i.e.  a compact Riemannian manifold with boundary of dimension $\ge 3$, which is conformally embedded in a product of the Euclidean real line and a simple manifold.  We prove the unique identifiability of the advection term of class $H^1\cap L^\infty$ and of class $H^{2/3}\cap C^{0,1/3}$  from the knowledge of the associated Dirichlet--to--Neumann map on the boundary of the manifold. This seems to be  the first global identifiability result for possibly discontinuous advection terms. 

\end{abstract}

\section{Introduction and statement of results}

Let $(M,g)$ be a   smooth compact oriented Riemannian manifold of dimension $n\ge 3$ with smooth boundary $\p M$.  Let $X\in L^{\infty}(M,TM)$ be a real vector field. In this paper we shall be concerned with an inverse boundary problem for the advection diffusion operator given by 
\[
L_{X}=-\Delta+X. 
\]
Specifically, our main focus is on establishing a global uniqueness result in the low regularity setting, allowing the advection term $X$ to be discontinuous.   Let us now proceed to introduce the precise assumptions and state the main results of this paper.  

Let $u\in H^1(M^0)$ be a solution to 
\begin{equation}
\label{eq_int_2}
L_{X}u=0 \quad \text{in}\quad \mathcal{D}'(M^0).   
\end{equation}
Here and in what follows $H^s(M^0)$, $s\in \R$, is the standard Sobolev space on $M^0$, see \cite[Chapter 4]{Taylor_book_1}, and $M^0=M\setminus\p M$ stands for the interior of $M$. We also let $\nu$ be the unit outer normal to the boundary of $M$. 
We shall define the trace of the normal derivative $\p_{\nu}u\in H^{-1/2}(\p M)$ as follows. Let $\varphi\in H^{1/2}(\p M)$. Then letting $v\in H^1(M^0)$ be a continuous extension of $\varphi$, we set 
\begin{equation}
\label{eq_int_2_2}
\langle \p_\nu u, \varphi \rangle_{H^{-1/2}(\p M)\times H^{1/2}(\p M)}=\int_{M} \langle \nabla u, \nabla v\rangle dV  +\int_M X(u)v dV,
\end{equation}
where $dV$ is the Riemannian volume element on $M$. 
As $u$ satisfies \eqref{eq_int_2}, the definition of the trace $\p_\nu u$ on $\p M$ is independent of the choice of an extension $v$ of $\varphi$.

For $f\in H^{1/2}(\p M)$, the Dirichlet problem, 
\begin{equation}
\label{eq_int_2_Dirichlet}
\begin{aligned}
&L_{X}u=0 \quad \text{in}\quad \mathcal{D}'(M^0),\\  
&u|_{\p M}=f,
\end{aligned}
\end{equation}
has the unique solution $u\in H^1(M^0)$, see \cite[Chapter 3, Section 8.2]{Aubin_book}. Thus,  we can define the Dirichlet--to--Neumann map,
\begin{equation}
\label{eq_int_2_2_Dirichlet_to_Neumann}
\Lambda_X:  H^{1/2}(\p M)\to H^{-1/2}(\p M), \quad f\mapsto \p_\nu u|_{\p M}. 
\end{equation}
The inverse boundary problem for the advection diffusion equation that we are interested in is to determine  the vector field $X$ from the knowledge of the Dirichlet--to--Neumann map $\Lambda_X$.  

This problem was studied extensively in the Euclidean case, see \cite{Cheng_Nakamura_Somersalo}, \cite{Knudsen_Salo}, \cite{Salo_diss}, \cite{Pohjola}. The sharpest result in terms of the regularity of the advection term is due to \cite{Pohjola}, showing the global uniqueness in the inverse boundary problem for the advection diffusion equation, with a H\"older continuous advection term, with the H\"older exponent in the range $(2/3,1]$.

It turns out that the inverse boundary problem for the advection diffusion  equation can be reduced to an inverse boundary problem for the magnetic Schr\"odinger equation, at least formally.  This connection has been exploited in the aforementioned works in the Euclidean case, and it can also be used to establish some global uniqueness results on manifolds. We shall therefore proceed now to recall the inverse boundary problem for the magnetic Schr\"odinger operator and to use this link to derive some preliminary global uniqueness results for the advection diffusion problem. 

Let  $A\in L^\infty (M,T^*M)$ and $q\in L^\infty(M,\C)$. The magnetic Schr\"odinger operator is defined by
\begin{equation}
\label{eq_int_1_magnetic}
\begin{aligned}
L_{A,q}u&=(d_{\overline{A}}^*d_A+q)u\\
&= -\Delta u +i d^*(Au)- i \langle A,du\rangle + (\langle A, A\rangle+q)u,\quad u\in C^\infty_0(M^0).
\end{aligned}
\end{equation}
Here $d$ is the de Rham differential, $d_A=d+i A$, and $d^*$, $d_A^*$ are the formal $L^2$--adjoints of $d$ and $d_A$, respectively. 

Let $u\in H^1(M^0)$ be a solution to 
\begin{equation}
\label{eq_int_3_mag}
L_{A,q} u=0
\end{equation}
in $\mathcal{D}'(M^0)$. Associated to $u$ is  the trace of the magnetic normal derivative $\langle d_Au , \nu \rangle\in H^{-1/2}(\p M)$ defined as follows, 
\begin{equation}
\label{eq_trace_normal_int_mag}
\begin{aligned}
\langle \langle d_A u, \nu \rangle, \varphi \rangle_{H^{-\frac{1}{2}}(\p M)\times H^{\frac{1}{2}}(\p M)}&:=(d_A u,d_{\overline{A}} \overline{v})_{L^2(M)}+(qu, \overline{v})_{L^2(M)}\\
&=\int_M \langle du+i A u, dv-iA v \rangle dV +\int_M q uv dV,
\end{aligned}
\end{equation}
where  $\varphi\in H^{1/2}(\p M)$ and $v\in H^1(M^0)$  is  a continuous extension of $\varphi$.

The set of the Cauchy data for solutions of the magnetic Schr\"odinger equation is given by 
\[
C_{A,q}:=\{ (u|_{\p M}, \langle d_A u, \nu \rangle_g|_{\p M}):  u\in H^1(M^0) \text{ satisfies } L_{A,q} u=0\text{ in } \mathcal{D}'(M^0) \}. 
\]
The inverse boundary value problem for the magnetic Schr\"odinger operator  is to
determine $A$ and $q$ in  $M$ from the knowledge of the set of the Cauchy data $C_{A,q}$.  A well-known feature of this problem is that there is an obstruction to uniqueness given by the following class of gauge transformations, see \cite{DKSaloU_2009}. Let $F\in W^{1,\infty}(M^0)$ be a non-vanishing function. For any $u\in H^1(M^0)$,  we have
\[
(F^{-1}\circ L_{A,q}\circ F)u=L_{A -i F^{-1}dF,q}u, 
\]
in $\mathcal{D}'(M^0)$. Furthermore, if  $F|_{\p M}=1$, a computation using \eqref{eq_trace_normal_int_mag} shows that 
\[
C_{A,q}=C_{A -i F^{-1}dF,q}.
\] 
Hence, from the knowledge of the set of the Cauchy data $C_{A,q}$ one may only hope to recover uniquely the magnetic field $dA$ and electric potential $q$ in $M$. 

The fundamental work \cite{DKSaloU_2009} initiated the study of inverse boundary problems on a class of compact Riemannian  manifolds with boundary, called admissible manifolds, the definition of which we shall now recall.

\begin{defn}
A  compact Riemannian manifold $(M,g)$ of dimension $n\ge 3$ with boundary $\p M$ is called admissible if there exists  an $(n-1)$--dimensional compact simple manifold $(M_0,g_0)$ such that $M\subset\R\times M_0$  and  $g=c(e\oplus g_0)$ where $e$ is the Euclidean metric on $\R$ and $c$ is a positive smooth function on $M$. 
\end{defn}

\begin{defn}
A compact manifold $(M_0, g_0)$ with boundary is called simple if for any $p\in M_0$, the exponential map $\exp_p$ with  its maximal domain of definition in $T_p M_0$ is a diffeomorphism onto $M_0$, and if $\p M_0$ is strictly convex in the sense that the second fundamental form of $\p M_0\hookrightarrow M_0$ is positive definite.  
\end{defn}

The inverse boundary problem for the magnetic Schr\"odinger operator on admissible manifolds was studied in  \cite{DKSaloU_2009} and the global uniqueness was established for $C^\infty$ electric and magnetic potentials.  The problem of weakening the regularity assumptions on the potentials in the context of admissible manifolds was addressed in \cite{DKSalo_2013} when $A=0$ and $q\in L^{\frac{n}{2}}(M)$, and in our recent work \cite{Krup_Uhlmann_magnet_manifolds} when $A\in L^\infty$ and $q\in L^\infty$. Let us state one of the main results of \cite{Krup_Uhlmann_magnet_manifolds}. 

\begin{thm}
\label{thm_old}
Let $(M,g)$ be admissible. Let $A^{(1)}, A^{(2)}\in L^\infty(M, T^*M)$ be complex valued $1$-forms, and $q^{(1)},q^{(2)}\in L^\infty(M,\C)$. If $C_{A^{(1)},q^{(1)}}=C_{A^{(2)},q^{(2)}}$, then $dA^{(1)}=dA^{(2)}$ and $q^{(1)}=q^{(2)}$. 
 \end{thm}

Let $X\in W^{1,\infty}(M^0, TM^0)$ and let $X^\flat=g_{jk}X^j dx^k$ be the corresponding one form. 
A direct computation using \eqref{eq_int_1_magnetic} shows that 
\[
L_{X}=L_{A,q},
\]
where
\begin{equation}
\label{eq_int_reduction_mag}
A=\frac{iX^\flat}{2}, \quad q=\frac{1}{4} \langle X, X\rangle-\frac{1}{2} \div(X). 
\end{equation}
Furthermore, if $u\in H^1(M^0)$ is a solution to $L_{X}u=L_{A,q}u=0$, then for any $\varphi \in H^{\frac{1}{2}}(\p M)$, we have
\[
\langle \langle d_A u, \nu \rangle, \varphi \rangle_{H^{-\frac{1}{2}}(\p M)\times H^{\frac{1}{2}}(\p M)}=\langle \p_\nu u, \varphi  \rangle_{H^{-\frac{1}{2}}(\p M)\times H^{\frac{1}{2}}(\p M)}-\frac{1}{2} \int_{\p M} \langle X,\nu \rangle u \varphi dS. 
\]
Hence, if $\Lambda_{X^{(1)}}=\Lambda_{X^{(2)}}$ then by Corollary \ref{cor_boundary_rec} we have 
$X^{(1)}=X^{(2)}$ on $\p M$,  and thus, $C_{A^{(1)},q^{(1)}}=C_{A^{(2)},q^{(2)}}$. Similarly to the Euclidean case, see \cite[Theorem 1.10]{Salo_diss}, the reduction described above allows us to obtain the following result as  a consequence of  Theorem \ref{thm_old}. This result can be viewed as an extension of \cite[Theorem 1.10]{Salo_diss} to the setting of admissible manifolds.    
\begin{thm}
\label{thm_conv_1}
Let $(M,g)$ be admissible and let $X^{(1)}, X^{(2)}\in W^{1,\infty}(M^0, TM^0)$ be real vector fields. If $\Lambda_{X^{(1)}} =\Lambda_{X^{(2)}}$ then $X^{(1)}=X^{(2)}$. 
 \end{thm}

A natural question is whether one can weaken the regularity of the advection term in Theorem \ref{thm_conv_1} in order to reach some discontinuous advection terms. To this end, let us state the main result of the present paper. 
\begin{thm}
\label{thm_main}
Let $(M,g)$ be admissible and let $X^{(1)}, X^{(2)}\in (H^{1}\cap L^\infty)(M^0, TM^0)$ be real  vector fields.  If $\Lambda_{X^{(1)}} =\Lambda_{X^{(2)}}$ then $X^{(1)}=X^{(2)}$. 
\end{thm} 

To the best of our knowledge this is  the first global identifiability result for possibly discontinuous advection terms.

\textbf{Remark 1.} As explained in \cite{DKSaloU_2009}, smooth bounded domains in the Euclidean space provide examples of admissible manifolds. Hence the  result of Theorem \ref{thm_main} is valid in the Euclidean case. To the best of our knowledge this is a new result even in this setting.  

\textbf{Remark 2.} Let $\Omega\subset \R^n$, $n\ge 3$, be a smooth bounded domain and let $\sigma:\Omega\to (0,\infty)$.   The conductivity equation,
\[
\nabla \cdot (\sigma \nabla u)=0\quad  \text{in}\quad  \Omega, 
\]
 can be rewritten in the following form,
\[
-\Delta u+X\cdot \nabla u=0 \quad  \text{in}\quad  \Omega,
\] 
 where $X=-\nabla \log\sigma$. In the inverse conductivity problem of Calder\'on, one wishes to determine the conductivity $\sigma$ from the associated Dirichlet--to--Neumann map, see \cite{Uhlmann_review}. We note that the inverse boundary problem for the advection diffusion equation is more general and consequently more difficult than the Calder\'on problem, since here we are aiming to determine a general vector field rather than a gradient one. Furthermore, if $\sigma\in W^{1,\infty}\cap H^{3/2}$ then $X\in H^{1/2}\cap L^\infty$ and we know from \cite[Proposition 2.5]{Krup_Uhlmann_calderon} that the method of $L^2$ Carleman estimates with a gain of two derivatives allows one to construct complex geometric optics solutions of the conductivity equation having the form,
\[
u(x,h)=\gamma^{-1/2}e^{\frac{x\cdot \zeta}{h}}(1+r(x,h)), \quad \zeta\in \C^n, \quad \zeta\cdot\zeta=0,\quad 0<h\ll 1,
\] 
with the remainder term $\|r\|_{H^1_{\text{scl}}}=o(h^{1/2})$ as $h\to 0$. Such solutions are sufficient to recover a conductivity of class $W^{1,\infty}\cap H^{3/2+\varepsilon}$, $\varepsilon>0$,  from the Dirichlet--to--Neumann map, see \cite{Krup_Uhlmann_calderon}.  On the other hand, for a general vector field $X\in H^{1/2}\cap L^\infty$, it seems that the same method only allows one to construct complex geometric optics solutions, with the remainder term enjoying the estimate 
$\|r\|_{H^1_{\text{scl}}}=o(h^{1/4})$,  see Proposition \ref{prop_cgo-admiss} below, which is not strong enough to solve the inverse problem in this case. A rough reason for this is that in the case of the general advection diffusion equation, the vector field contributes to the transport equation for the amplitude, see \eqref{eq_cgo_4}, leading to worse estimates for the remainder. 

\textbf{Remark 3.} It seems that working in the scale of spaces $X \in H^s \cap L^\infty$, the condition $s=1$ corresponds to the largest space for which the inverse boundary problem for the advection diffusion equation can be solved by means of the techniques of $L^ 2$ Carleman estimates with a gain of two derivatives.

\textbf{Remark 4.}  Let $C^{0,\gamma}(M, E)$, $0<\gamma\le 1$, be the space of  H\"older continuous sections of a vector bundle $E$ over $M$, see  \cite[p. 42]{Hormander_1976}.  Similarly to \cite{Pohjola}, working in the scale of H\"older spaces, one should be able to recover the advection term $X\in C^{0, \frac{2}{3}+\varepsilon}(M, TM)$,   $\varepsilon>0$.  It seems that  $\gamma=2/3$ is a natural threshold for the inverse boundary problem for the advection diffusion equation to be solved by means of the techniques of $L^2$ Carleman estimates. We decided not to pursue   this direction since our main interest is in getting some accurate results for possibly discontinuous vector fields.  However, in Theorem \ref{thm_main} a further decrease in the Sobolev regularity, down to $s=2/3$, is possible at the expense of demanding some additional H\"older regularity for $X$. In this direction, we have the following result. 

\begin{thm}
\label{thm_main_2}
Let $(M,g)$ be admissible and let $X^{(1)}, X^{(2)}\in H^{\frac{2}{3}} (M^0, TM^0) \cap C^{0,\frac{1}{3}} (M, TM)$ be real  vector fields.  If $\Lambda_{X^{(1)}} =\Lambda_{X^{(2)}}$ then $X^{(1)}=X^{(2)}$. 
\end{thm}

Let us now proceed to describe the main ideas in the proof of Theorem \ref{thm_main}, the proof of Theorem \ref{thm_main_2} being quite similar. When proving Theorem \ref{thm_main}, to keep track of the regularity needed we let $X \in H^s \cap L^\infty$, $0\le s\le 1$, and  work with the advection diffusion equation directly, rather than reducing it to the magnetic Schr\"odinger equation. A crucial ingredient in the proof  is the construction of complex geometric optics solutions to an equation of the form
\begin{equation}
\label{eq_intro_P}
P_{X,Y, q}u:=-\Delta u +Xu +\div (Y)u +qu=0,
\end{equation}
with $X,Y\in (H^s\cap L^\infty)(M^0, TM^0)$, $0\le s\le 1$, and $q\in L^\infty(M,\C)$.  The form of the equation is general enough to comprise both the operator $L_X$ as well as its adjoint. The main difficulty here is that the operator $P_{X,Y,q}$ has singular coefficients and to overcome this difficulty, similarly to our works \cite{Krup_Uhlmann_magnet},  \cite{Krup_Uhlmann_magnet_manifolds}, we shall rely on a Carleman estimate with a gain of two derivatives, established in \cite{Krup_Uhlmann_magnet_manifolds} on a compact manifold admitting a limiting Carleman weight. When constructing the complex geometric optics solutions we also make use of a smoothing argument, approximating an $H^s\cap L^\infty$ vector field $X$ by a sequence of smooth vector fields $X_{\tau}$ such that 
\[
\|X-X_\tau\|_{L^2(M)}=o(\tau^s), \quad \tau\to 0. 
\]
When $s=1$, in order to obtain an estimate for the remainder in the complex geometric optics solutions that is sufficiently precise for our purposes, we have to exploit the Gagliardo--Nirenberg inequality. Another crucial ingredient in the proof is the boundary reconstruction of $X\in H^1 \cap L^{\infty}$, in the Sobolev trace sense, and to this end we adapt the arguments of \cite{Brown_2001}, see also \cite{Brown_Salo_2006}, combining them with  some results of  \cite{Krup_Uhlmann_magnet_manifolds}.  At the final stages of the proof, we make use of results related to the injectivity of the attenuated ray transform, acting in the space of  compactly supported distributions and forms on a simple manifold, established in \cite{DKSalo_2013},  \cite{Assylbekov_Yang_2017}, and \cite{rodriguez}. 

The plan of the paper is as follows. In Section \ref{sec_Carleman_est} we recall the Carleman estimates with a gain of two derivatives of \cite{Krup_Uhlmann_magnet_manifolds}, and establish solvability results for the operator $P_{X,Y,q}$ given in \eqref{eq_intro_P}.  Complex geometric optics solutions for the equation \eqref{eq_intro_P} are constructed in Section \ref{sec_CGO}, and the proof of Theorem \ref{thm_main} is completed in Section \ref{sec_proof}.  Section \ref{sec_proof_thm_2} describes the modifications in the arguments needed to establish Theorem \ref{thm_main_2}. The boundary determination of the advection term is the subject of Appendix \ref{sec_boundary_rec}, and standard approximation estimates are collected in Appendix~\ref{sec_approx_est}.

\section{Carleman estimates and solvability results}

\label{sec_Carleman_est}

Let us start by recalling the Carleman estimate for the semiclassical Laplacian $-h^2\Delta$, $0<h\le 1$, with a gain of two derivatives, established in \cite{Krup_Uhlmann_magnet_manifolds} in the case of Riemannian manifolds  admitting limiting Carleman weights. This result is an extension of  \cite[Lemma 2.1]{Salo_Tzou_2009} obtained in the Euclidean case.

Let $(M,g)$ be a compact smooth Riemannian manifold with boundary.  Assume that $(M,g)$ is  embedded in a compact smooth  Riemannian  manifold $(N,g)$ without boundary of the same dimension, and let $U$ be open in $N$ such that $M\subset U$.

Let  $\varphi\in C^\infty(U,\R)$ and let us consider the conjugated operator 
\begin{equation}
\label{eq_Car_-2}
P_\varphi=e^{\frac{\varphi}{h}}(-h^2\Delta)e^{-\frac{\varphi}{h}}=-h^2\Delta -|\nabla \varphi|^2+2\langle \nabla \varphi, h\nabla\rangle +h\Delta\varphi,
\end{equation}
with the semiclassical principal symbol 
\begin{equation}
\label{eq_Car_1}
p_\varphi(x,\xi)=|\xi|^2-|d\varphi|^2+2i \langle \xi, d\varphi\rangle\in C^\infty(T^*U).
\end{equation}

Here and in what follows we use $\langle \cdot, \cdot \rangle$ and $|\cdot|$ to denote the Riemannian scalar product and norm both on the tangent and cotangent space.  

Following \cite{Kenig_Sjostrand_Uhlmann},  \cite{DKSaloU_2009}, we say that $\varphi\in C^\infty(U,\R)$ is a limiting Carleman weight for $-h^2\Delta$ on $(U,g)$ if $d\varphi\ne 0$ on $U$, and the Poisson bracket of $\Re p_\varphi$ and $\Im p_\varphi$ satisfies,
\[
\{\Re p_\varphi, \Im p_\varphi\}=0 \quad\text{when}\quad  p_\varphi=0. 
\]
We refer to \cite{DKSaloU_2009} for a characterization of Riemannian manifolds admitting limiting Carleman weights.

In what follows we equip the Sobolev space $H^s(N)$, $s\in\R$, with the natural semiclassical norm
\[
\|u\|_{H^s_{\text{scl}}(N)}=\|(1-h^2\Delta)^{\frac{s}{2}} u\|_{L^2(N)}. 
\]
 Our starting point is the following result  of \cite{Krup_Uhlmann_magnet_manifolds}. 
\begin{prop}
\label{prop_Carleman_est_gain_2}
Let $\varphi$ be a limiting Carleman weight for $-h^2\Delta$ on $(U,g)$ and let $\tilde \varphi=\varphi+\frac{h}{2\varepsilon}\varphi^2$. Then for all $0<h\ll \varepsilon\ll 1$ and $s\in \R$, we have
\begin{equation}
\label{eq_Car_for_laplace}
\frac{h}{\sqrt{\varepsilon}}\|u\|_{H^{s+2}_{\emph{\text{scl}}}(N)}\le C\|e^{\frac{\tilde \varphi }{h}}(-h^2\Delta)e^{-\frac{\tilde \varphi}{h}}  u\|_{H^{s}_{\emph{\text{scl}}}(N)}, \quad C>0,
\end{equation}
for all $u\in C_0^\infty(M^0)$. 
\end{prop}

We shall next state a Carleman estimate for a suitable first order perturbation of $-h^2\Delta$ which is needed in the proof of Theorem \ref{thm_main}.  Let $Y\in L^\infty(M,TM)$. Then  $\div(Y)\in H^{-1}(M^0) $ is given by 
\begin{equation}
\label{eq_int_2_dis}
\langle \div(Y), \varphi \rangle_{M^0}:=- \int Y(\varphi) dV,\quad \varphi\in C^\infty_0(M^0),
\end{equation}
where $\langle \cdot, \cdot\rangle_{M^0}$ is the distributional duality on $M^0$. We shall also view $\div (Y)$ as a multiplication operator,
\[
\div (Y): C^\infty_0(M^0)\to H^{-1}(M^0). 
\]

\begin{prop}
\label{prop_Carleman_for_convection}
Let $X, Y\in L^\infty(M,TM)$ be complex vector fields, and  let $q\in L^\infty(M, \C)$. Set 
\[
P_{X,Y,q}=-\Delta +X +\div(Y)+q: C^\infty_0(M^0)\to H^{-1}(N).
\]
Let $\varphi$ be a limiting Carleman weight for $-h^2\Delta$ on $(U,g)$. Then for all $0<h\ll 1$, we have
\begin{equation}
\label{eq_Car_for_convection}
h \|u\|_{H^{1}_{\emph{\text{scl}}}(N)}\le C\|e^{\frac{ \varphi }{h}}(h^2 P_{X,Y,q})e^{-\frac{\varphi}{h}}  u\|_{H^{-1}_{\emph{\text{scl}}}(N)}, \quad C>0,
\end{equation}
for all $u\in C_0^\infty(M^0)$. 
\end{prop} 
\begin{proof}
Let $\tilde \varphi=\varphi+\frac{h}{2\varepsilon}\varphi^2$ with $\varepsilon>0$ be such that $0<h\ll \varepsilon\ll 1$.  We have
\[
e^{\frac{\tilde \varphi }{h}} (h^2 X +h^2\div(Y)+h^2q)  e^{-\frac{\tilde \varphi}{h}} =h^2 X-h\bigg(1+\frac{h}{\varepsilon}\varphi \bigg)X(\varphi)+ h^2\div(Y)+h^2q.
\]
Therefore, for $u\in C^\infty_0(M^0)$, we get 
\begin{equation}
\label{eq_Car_for_convection_1}
\| e^{\frac{\tilde \varphi }{h}} h^2 X (e^{-\frac{\tilde \varphi}{h}} u) \|_{H^{-1}_{\text{scl}}(N)}\le \|  h^2 X(u)-h\bigg(1+\frac{h}{\varepsilon}\varphi \bigg)X(\varphi)  u \|_{L^2(N)}\le \mathcal{O}(h)\|u\|_{H^1_{\text{scl}}(N)}.
\end{equation}

In order to estimate $\|h^2\div(Y)u \|_{H^{-1}_{\text{scl}}(N)}$,  we shall use the following characterization of the semiclassical norm in the Sobolev space $H^{-1}(N)$, 
\[
\|v\|_{H^{-1}_{\text{scl}}(N)}=\sup_{0\ne \psi\in C^\infty(N)}\frac{|\langle v,\psi\rangle_N|}{\| \psi\|_{H^1_{\text{scl}}(N)}}. 
\]
Using \eqref{eq_int_2_dis}, for $0\ne \psi\in C^\infty(N)$, we get 
\[
|\langle h^2\div(Y)u,\psi\rangle_N|\le \int_N h^2|Y(u\psi)|dV\le \mathcal{O}(h)\|u\|_{H^{1}_{\text{scl}}(N)}\|\psi\|_{H^{1}_{\text{scl}}(N)},
\]
and therefore, 
\begin{equation}
\label{eq_Car_for_convection_2}
\|h^2\div(Y)u \|_{H^{-1}_{\text{scl}}(N)}\le \mathcal{O}(h)\|u\|_{H^{1}_{\text{scl}}(N)}.
\end{equation}
Finally, we have
\begin{equation}
\label{eq_Carleman_convection_pot}
\|h^2qu\|_{H^{-1}_{\text{scl}}(N)}\le \mathcal{O}(h^2)\|u\|_{H^{1}_{\text{scl}}(N)}.
\end{equation}

Hence, choosing $\varepsilon>0$ sufficiently small but fixed, i.e. independent of $h$, we obtain from \eqref{eq_Car_for_laplace} with $s=-1$ and \eqref{eq_Car_for_convection_1}, \eqref{eq_Car_for_convection_2}, and \eqref{eq_Carleman_convection_pot} that for all $h>0$ small enough and $u\in C_0^\infty(M^0)$, 
\[
\|e^{\frac{\tilde \varphi }{h}}(h^2P_{X,Y, q})e^{-\frac{\tilde \varphi}{h}}  u\|_{H^{-1}_{\text{scl}}(N)}\ge \frac{h}{C}\|u\|_{H^{1}_{\text{scl}}(N)},
\]
which implies \eqref{eq_Car_for_convection}.
\end{proof}

Notice that the formal $L^2$ adjoint of $P_{X,Y,q}$ is given by $P_{-\overline{X}, -\overline{X}+\overline{Y},\overline{q}}$. Using the fact that  if $\varphi$ is a limiting Carleman weight then so is $-\varphi$, we obtain the following solvability result, see  \cite{DKSaloU_2009} and \cite{Krup_Uhlmann_magnet} for the details. 
\begin{prop}
\label{prop_solvability}
Let $X, Y\in L^\infty(M,TM)$ be complex vector fields, and  let $q\in L^\infty(M, \C)$. Let  $\varphi$ be a limiting Carleman weight for $-h^2\Delta$ on $(U,g)$. 
If  $h>0$ is small enough, then for any $v\in H^{-1}(M^0)$, there is a solution $u\in H^1(M^0)$ of the equation 
\[
e^{\frac{\varphi}{h}}(h^2P_{X,Y,q})e^{-\frac{\varphi}{h}}  u=v \quad \text{in}\quad M^0,
\]
which satisfies 
\[
\|u\|_{H^1_{\emph{\text{scl}}}(M^0)}\le \frac{C}{h} \|v\|_{H^{-1}_{\emph{\text{scl}}}(M^0)}.
\]
\end{prop}

\section{Complex geometric optics solutions}

\label{sec_CGO}

Let $(M,g)$ be an admissible manifold.  
 Then  $(M,g)$ is isometrically  embedded in $(\R\times M_0, c(e\oplus g_0))$ for some compact simple $(n-1)$--dimensional manifold $(M_0,g_0)$ and some $0<c\in C^\infty(\R\times M_0)$. Assume, 
after replacing $M_0$ by a slightly larger simple manifold if needed, that for some simple manifold $(D,g_0)\subset\subset (M_0^0,g_0)$ one has  
\begin{equation}
\label{eq_cgo_mnfl_0_1}
(M,g)\subset\subset (\R\times D^0, c(e\oplus g_0))\subset\subset (\R\times M_0^0, c(e\oplus g_0)). 
\end{equation}

Let $X\in (H^s\cap L^\infty)(M^0,TM^0)$,  $0\le s\le 1$. Arguing as in \cite[Section 2.2]{Krup_Uhlmann_calderon}, we shall extend $X$ to  a compactly supported vector field in $(H^s\cap L^\infty)(\R\times M_0^0, T(\R\times M_0^0))$ and we shall denote the extension by the same letter.  
Using a partition of unity argument together with a regularization in each coordinate patch, we get the following result in view of Proposition \ref{prop_app_R_n}. 

\begin{prop}
\label{prop_approximation}
Let $X\in H^s(\R\times M_0^0, T(\R\times M_0^0))$, $0\le s\le 1$. There exists a family  $X_\tau\in C^\infty_0(\R\times M_0^0, T(\R\times M_0^0))$, $\tau>0$,  such that 
\begin{equation}
\label{eq_Car_20_0}
\|X-X_\tau\|_{L^2}=o(\tau^s),\quad \tau\to 0,
\end{equation}
and 
\begin{equation}
\label{eq_Car_20}
\begin{aligned}
\|X_\tau\|_{L^2}=\mathcal{O}(1), \quad   \|\nabla X_\tau\|_{L^2}=\begin{cases}o(\tau^{s-1}), & 0\le s<1,\\
\mathcal{O}(1), & s=1,
\end{cases}
\quad
\|\nabla^2 X_\tau\|_{L^2}=o(\tau^{s-2}).
\end{aligned}
\end{equation}
If furthermore, $X\in L^\infty(\R\times M_0, T(\R\times M_0))$, we have
\begin{equation}
\label{eq_Car_20_L_infty}
\begin{aligned}
\|X_\tau\|_{L^\infty}=\mathcal{O}(1), \quad
\|\nabla X_\tau\|_{L^\infty}=\mathcal{O}(\tau^{-1}),\quad  \tau\to 0.  
\end{aligned}
\end{equation}
\end{prop}

We have global coordinates $ x=(x_1,x')$ on $\R\times M_0$ in which the metric $g$ has the form
\begin{equation}
\label{eq_cgo_metric_0}
g(x)=c(x)\begin{pmatrix}
1& 0\\
0 & g_0(x')
\end{pmatrix},
\end{equation} 
where $c>0$ and $g_0$ is a simple metric on $M_0$. Then the globally defined function 
\begin{equation}
\label{eq_cgo_metric_0_1}
\varphi(x)=x_1
\end{equation}
 on $M$
is a limiting Carleman weight, see \cite{DKSaloU_2009}.

Let $X, Y\in (H^s\cap L^\infty)(M^0, TM^0))$,  $0\le s\le 1$, be complex vector fields and let $q\in L^\infty(M, \C)$. 
We shall construct solutions to the equation
\begin{equation}
\label{eq_cgo_1}
P_{X,Y,q}u=-\Delta u+Xu +\div(Y)u+qu=0\quad\text{in}\quad M^0
\end{equation}
 of the form
\begin{equation}
\label{eq_cgo_2}
u=e^{-\frac{\rho}{h}}(a+r_0),
\end{equation}
where $\rho=\varphi+i\psi$ is the complex phase with $\varphi$ is given by \eqref{eq_cgo_metric_0_1},  $\psi\in C^\infty(M,\R)$, $a\in C^\infty(M,\C)$ is an amplitude, obtained by a WKB construction, and $r_0$ is a remainder term.  A direct computation shows that  
\[
e^{\frac{\rho}{h}} \circ (-h^2\Delta) \circ e^{-\frac{\rho}{h}}=-h^2\Delta +h\Delta \rho+2h \nabla\rho -|\nabla \rho|^2,
\]
where $\nabla\rho$ is a complex vector field, and $|\nabla \rho|^2=\langle \nabla \rho,\nabla \rho\rangle $ is computed using the bilinear extension of the Riemannian scalar product to the complexified tangent bundle.  
Furthermore, 
\[
e^{\frac{\rho}{h}} h^2 (X+ \div(Y)+q) e^{-\frac{\rho}{h}}= h^2X-hX(\rho) +h^2\div(Y)+h^2q,
\]
and therefore, 
\begin{align*}
e^{\frac{\rho}{h}} h^2P_{X,Y,q} (e^{-\frac{\rho}{h}}u)=&-h^2\Delta u  +h(\Delta \rho) u+2h \nabla\rho(u) -|\nabla \rho|^2 u +h^2X(u)-hX(\rho)u\\
&+ h^2\div(Y)u+h^2qu.
\end{align*}
In order for \eqref{eq_cgo_2} to be a solution to \eqref{eq_cgo_1}, following the WKB method, we require that $\rho$ satisfies the eikonal equation,  
\begin{equation}
\label{eq_cgo_3}
|\nabla \rho|^2=0,
\end{equation}
and the amplitude $a$ satisfies  the regularized transport equation,
\begin{equation}
\label{eq_cgo_4}
-X_\tau(\rho)a  +2 \nabla\rho (a)+  (\Delta \rho) a=0. 
\end{equation}
The remainder term $r_0$ will be then determined by solving the equation, 
\begin{equation}
\label{eq_cgo_5}
e^{\frac{\rho}{h}} h^2P_{X,Y,q} (e^{-\frac{\rho}{h}}r_0)=-( -h^2\Delta a+ h^2X(a)-h(X-X_\tau)(\rho)a + h^2\div(Y)a+h^2qa). 
\end{equation}

Recalling the definition of  $\varphi$ given in \eqref{eq_cgo_metric_0_1}, we see that \eqref{eq_cgo_3} becomes a pair of equations for $\psi$,
\begin{equation}
\label{eq_cgo_6}
|\nabla \psi|^2=|\nabla \varphi|^2, \quad \langle \nabla \varphi, \nabla \psi\rangle =0. 
\end{equation}
By  \eqref{eq_cgo_metric_0} and \eqref{eq_cgo_metric_0_1}, we get 
\begin{equation}
\label{eq_cgo_7}
\nabla \varphi =\frac{1}{c}\p_{x_1}, \quad |\nabla \varphi|^2=\frac{1}{c}. 
\end{equation}
We conclude from \eqref{eq_cgo_6} and \eqref{eq_cgo_7} that 
\begin{equation}
\label{eq_cgo_8}
|\nabla \psi|^2=\frac{1}{c}, \quad \p_{x_1}\psi=0. 
\end{equation}

Let $\omega\in D$ be a point such that $(x_1,\omega)\notin M$ for all $x_1\in \R$.  We have global coordinates on $M$ given by $x=(x_1,r,\theta)$, where $(r,\theta)$ are the polar normal coordinates in $(D,g_0)$ with center $\omega$, i.e. $x'=\exp_\omega^D(r\theta)$ where $r>0$ and $\theta\in \mathbb{S}^{n-2}$. Since $D$ is simple, the exponential map  $\exp_\omega^D$ takes its maximal domain in $T_\omega D$ diffeomorphically onto $D$.  It follows from \eqref{eq_cgo_metric_0} that  in these coordinates  the metric $g$ has the form,
\begin{equation}
\label{eq_cgo_8_metrix}
g(x_1,r,\theta)=c(x_1,r,\theta) \begin{pmatrix}
1& 0& 0\\
0& 1& 0\\
0& 0& m(r,\theta)
\end{pmatrix},
\end{equation}
where $m$ is a smooth positive definite $(n-2)\times (n-2)$ matrix.  The eikonal equation \eqref{eq_cgo_8} has therefore a global $C^\infty$ solution 
\[
\psi(x)=\psi_\omega(x)=r, 
\]
and we get 
\[
\rho=x_1+i r, \quad  \nabla \rho=\frac{2}{c}\overline{\p},
 \]
where 
\[
\overline{\p}=\frac{1}{2}(\p_{x_1}+i\p_{r}). 
\]
We have
\[
\Delta \rho=|g|^{-1/2}\p_{x_1} \bigg(\frac{|g|^{1/2}}{c}\bigg)+i |g|^{-1/2}\p_{r} \bigg(\frac{|g|^{1/2}}{c}\bigg)=\frac{1}{c}\overline{\p} \log\bigg(\frac{|g|}{c^2}\bigg).
\]
Writing $X_\tau=(X_\tau)_1\p_{x_1}+(X_\tau)_r\p_r+(X_\tau)_\theta\p_\theta$, we see that the transport equation \eqref{eq_cgo_4} has the form,
\begin{equation}
\label{eq_cgo_9}
4\overline{\p }a + \bigg(\overline{\p} \log\bigg(\frac{|g|}{c^2}\bigg)\bigg) a=  ((X_\tau)_1+i (X_\tau)_r)ca. 
\end{equation}

Following \cite{DKSaloU_2009}, we choose a solution $a$ of \eqref{eq_cgo_9} in the form, 
\begin{equation}
\label{eq_cgo_10_ampl}
a=|g|^{-1/4} c^{1/2} e^{\Phi_\tau} a_0(x_1,r)b(\theta),
\end{equation}
where $\Phi_\tau$ solves the equation 
\begin{equation}
\label{eq_cgo_10}
\overline{\p}\Phi_\tau =\frac{1}{4}((X_\tau)_1 +i (X_\tau)_r)c,
\end{equation}
$a_0$ is  a non-vanishing holomorphic function, 
\[
\overline{\p}a_0=0,
\]
and $b(\theta)$ is smooth.  The inhomogeneous $\overline{\p} $ equation \eqref{eq_cgo_10} is given in the global coordinates $(x_1,r)$, and  we can take
\begin{equation}
\label{eq_cgo_11}
\Phi_\tau(x_1,r,\theta) = \frac{1}{4} \frac{1}{\pi (x_1+i r)} * ((X_\tau)_1 +i (X_\tau)_r) c,
\end{equation}
with $*$ denoting the convolution in the variables $(x_1,r)$ and $X_\tau(\cdot,\cdot,\theta)$ being viewed as a compactly supported smooth vector field in the complex $x_1+i r$ plane. 

By \eqref{eq_cgo_11} and \eqref{eq_Car_20},  we get  
\begin{equation}
\label{eq_cgo_12}
\begin{aligned}
&\|\Phi_\tau\|_{L^2(M)}=\mathcal{O}(1), \ \|\nabla \Phi_\tau\|_{L^2(M)}=\begin{cases}o(\tau^{s-1}), & 0\le s<1,\\
\mathcal{O}(1), & s=1,
\end{cases}\\
 &\|\Delta \Phi_\tau\|_{L^2(M)}=o(\tau^{s-2}),
\end{aligned}
\end{equation}
as $\tau\to 0$. 
Also by \eqref{eq_Car_20_L_infty} we have
\begin{equation}
\label{eq_cgo_12_L_infty}
\begin{aligned}
\|\Phi_\tau\|_{L^\infty(M)}=\mathcal{O}(1), \quad
\|\nabla \Phi_\tau\|_{L^\infty(M)}=\mathcal{O}(\tau^{-1}),\quad \tau\to 0.  
\end{aligned}
\end{equation}
Setting 
\[
\Phi(x_1,r,\theta) =  \frac{1}{4} \frac{1}{\pi (x_1+i r)} * (X_1 +i X_r) c \in  L^\infty(M), 
\]
using Young's inequality, \eqref{eq_cgo_11} and \eqref{eq_Car_20_0}, we get 
\begin{equation}
\label{eq_cgo_13}
\|\Phi-\Phi_\tau\|_{L^2(M)}=o(\tau^s), \quad \tau\to 0. 
\end{equation}

It follows from \eqref{eq_cgo_10_ampl}, \eqref{eq_cgo_12} and \eqref{eq_cgo_12_L_infty} that 
\begin{equation}
\label{eq_cgo_12_ampl_estimates}
\begin{aligned}
&\|a\|_{L^2(M)}=\mathcal{O}(1), \quad  \|\nabla a\|_{L^2(M)}=\begin{cases}o(\tau^{s-1}), & 0\le s<1,\\
\mathcal{O}(1), & s=1,
\end{cases},\\
& \|\Delta a\|_{L^2(M)}=o(\tau^{s-2}),\ 0\le s\le 1,\\
&\|a\|_{L^\infty(M)}=\mathcal{O}(1), \quad \|\nabla a\|_{L^\infty(M)}=\mathcal{O}(\tau^{-1}),
\end{aligned}
\end{equation}
as $\tau\to 0$.  The only estimate that should be explained in detail  is the following one, 
\begin{equation}
\label{eq_cgo_12_ampl_estimates_proff}
 \|\Delta a\|_{L^2(M)}=o(\tau^{-1}),  
 \end{equation}
 when $s=1$. 
Indeed, it follows from  \eqref{eq_cgo_10_ampl} that 
\[
a=fe^{\Phi_\tau},\quad f=|g|^{-1/4}c^{1/2}a_0b\in C^\infty(M),
\]
and therefore, 
\[
\Delta a=f e^{\Phi_\tau} \langle \nabla \Phi_\tau,\nabla\Phi_\tau\rangle+ f e^{\Phi_\tau}\Delta \Phi_\tau+ 2\langle \nabla f, e^{\Phi_\tau} \nabla \Phi_\tau\rangle+ e^{\Phi_\tau}\Delta f. 
\]
Hence, using the Gagliardo--Nirenberg inequality,
\[
\|\nabla u\|_{L^4(M)}\le C\|u\|_{L^\infty(M)}^{1/2}\|u\|_{H^2(M^0)}^{1/2},
\]
valid for $u\in (L^\infty\cap H^2)(M^0)$, see \cite[page 101]{Alinhac_Gerard}, \cite[p. 313]{Brezis_book}, and \eqref{eq_cgo_12}, \eqref{eq_cgo_12_L_infty}, we get 
\[
\|\Delta a\|_{L^2(M)}\le \mathcal{O}(1)(\|\nabla \Phi_\tau\|^2_{L^4(M)}+o(\tau^{-1})+1)=o(\tau^{-1}),  
\]
showing \eqref{eq_cgo_12_ampl_estimates_proff}.

Now let us solve the equation \eqref{eq_cgo_5} for the remainder term $r_0$. Consider the right hand side of  \eqref{eq_cgo_5},
\begin{equation}
\label{eq_cgo_14_v}
v=-( -h^2\Delta a+ h^2X(a)-h(X-X_\tau)(\rho)a + h^2\div(Y)a+h^2qa)\in H^{-1}(M^0).  
\end{equation}
First, using \eqref{eq_cgo_12_ampl_estimates} and \eqref{eq_Car_20_0},  we have
\begin{equation}
\label{eq_cgo_14_v_first}
\begin{aligned}
&\|h^2\Delta a\|_{H_{\text{scl}}^{-1}(M^0)}\le h^2\|\Delta a\|_{L^2(M)}\le h^2o(\tau^{s-2}), \\
& \|h^2X(a)\|_{H_{\text{scl}}^{-1}(M^0)}\le h^2\|X(a)\|_{L^2(M)}\le \begin{cases}h^2o(\tau^{s-1}), & 0\le s<1,\\
h^2\mathcal{O}(1), & s=1,
\end{cases},\\
&\|h(X-X_\tau)(\rho)a\|_{H_{\text{scl}}^{-1}(M^0)}\le h\|(X-X_\tau)(\rho)a\|_{L^2(M)}\le ho(\tau^s),\\
&\|h^2 qa\|_{H_{\text{scl}}^{-1}(M^0)}\le h^2\|qa\|_{L^2(M)}\le \mathcal{O}(h^2). 
\end{aligned}
\end{equation}
Letting  $0\ne \phi\in C^\infty_0(M^0)$, and using \eqref{eq_cgo_12_ampl_estimates} and Proposition \ref{prop_approximation}, we obtain that 
\begin{equation}
\label{eq_cgo_14_v_second}
\begin{aligned}
|\langle h^2\div(Y)a, &\phi \rangle_{M^0}| \le h^2 |\langle \div(Y_\tau)a, \phi \rangle_{M^0}|+ h^2\int_{M^0}  |(Y-Y_\tau)(a \phi)| dV\\
&\le \mathcal{O}(h^2) \| \div(Y_\tau)\|_{L^2(M)} \|a\|_{L^\infty(M)}\|\phi\|_{L^2(M)}\\
&+ \mathcal{O}(h^2)  \|Y-Y_\tau\|_{L^2(M)} (\|\nabla a\|_{L^\infty(M)}\|\phi\|_{L^2}+ \|a\|_{L^\infty(M)}\|\nabla \phi\|_{L^2(M)})\\
&\le  \begin{cases}(h^2o(\tau^{s-1})  +ho(\tau^s))\|\phi\|_{H^1_{\text{scl}}(M^0)}, & 0\le s<1,\\
(h^2\mathcal{O}(1) + h^2o(1)+ho(\tau)) \|\phi\|_{H^1_{\text{scl}}(M^0)}, & s=1.
\end{cases}
\end{aligned}
\end{equation}
Combining \eqref{eq_cgo_14_v}, \eqref{eq_cgo_14_v_first}, and \eqref{eq_cgo_14_v_second}, we get 
\[
\|v\|_{H^{-1}_{\text{scl}}(M^0)}\le h^2o(\tau^{s-2})  +ho(\tau^s), \quad  0\le s\le 1. 
\]
Hence, choosing $\tau=h^{1/2}$, we see that $\|v\|_{H^{-1}_{\text{scl}}(M^0)}=o(h^{1+s/2})$. Thus,  
by Proposition \ref{prop_solvability}, for all $h>0$ small enough, there exists a solution $r_0\in H^1(M^0)$ of    \eqref{eq_cgo_5} which satisfies $\|r_0\|_{H^{1}_{\text{scl}}(M^0)}=o(h^{s/2})$ as $h\to 0$. 

The discussion above can be summarized in the following proposition.
\begin{prop}
\label{prop_cgo-admiss}
Assume that $(M,g)$ satisfies \eqref{eq_cgo_mnfl_0_1} and \eqref{eq_cgo_metric_0}, and let $X, Y\in(H^s\cap L^\infty)(M^0, TM^0)$, $0\le s\le 1$, be complex vector fields and $q\in L^\infty(M, \C)$.  Let $\omega\in D$ be such that $(x_1,\omega)\notin M$ for all $x_1$, and let $(r,\theta)$ be the polar normal coordinates in $(D,g_0)$ with center $\omega$. Then for all $h>0$ small enough, there exists a solution $u\in H^1(M^0)$ to the equation 
\[
P_{X,Y,q} u=0\quad \text{in}\quad \mathcal{D}'(M^0)
\]
of the form
\[
u=e^{-\frac{1}{h}(x_1+ir )}(a +r_0 ), \quad a=|g|^{-1/4} c^{1/2} e^{\Phi_h} a_0(x_1,r)b(\theta),
\]
where $a_0$ is  a non-vanishing holomorphic function, $(\p_{x_1}+i\p_r)a_0=0$, and $b(\theta)$ is smooth. The function
$\Phi_h\in C^\infty(M)$ satisfies  
\begin{equation}
\label{eq_cgo_19}
\begin{aligned}
&\|\Phi_h\|_{L^2(M)}=\mathcal{O}(1), \ \|\nabla \Phi_h\|_{L^2(M)}=\begin{cases}o(h^{\frac{s-1}{2}}), & 0\le s<1,\\
\mathcal{O}(1), & s=1,
\end{cases}\\
 &\|\Delta \Phi_h\|_{L^2(M)}=o(h^{\frac{s-2}{2}}),
\end{aligned}
\end{equation}
and
\begin{equation}
\label{eq_cgo_19_L_infty}
\begin{aligned}
\|\Phi_h\|_{L^\infty(M)}=\mathcal{O}(1), \quad
\|\nabla \Phi_h\|_{L^\infty(M)}=\mathcal{O}(h^{-\frac{1}{2}}),
\end{aligned}
\end{equation}
and    
\begin{equation}
\label{eq_cgo_20}
\|\Phi-\Phi_h\|_{L^2(M)}=o(h^{\frac{s}{2}}), \quad h\to 0, 
\end{equation}
where 
\[
\Phi(x_1,r,\theta) = \frac{1}{4} \frac{1}{\pi (x_1+i r)} * (X_1 +i X_r)c,
\]
where $X=X_1 \p_{x_1}+X_r\p_r+X_\theta \p_\theta$.  Furthermore, $a$ satisfies 
\begin{equation}
\label{eq_cgo_19_ampl_estimates}
\begin{aligned}
&\|a\|_{L^2(M)}=\mathcal{O}(1), \quad  \|\nabla a\|_{L^2(M)}=\begin{cases}o(h^{\frac{s-1}{2}}), & 0\le s<1,\\
\mathcal{O}(1), & s=1,
\end{cases},\ \|\Delta a\|_{L^2(M)}=o(h^{\frac{s-2}{2}}),\\
&\|a\|_{L^\infty(M)}=\mathcal{O}(1), \quad \|\nabla a\|_{L^\infty(M)}=\mathcal{O}(h^{-\frac{1}{2}}).
\end{aligned}
\end{equation}
The remainder $r_0$ is such that $\|r_0\|_{H^{1}_{\emph{\text{scl}}}(M^0)}=o(h^{\frac{s}{2}})$ as $h\to 0$. 
\end{prop}

\section{Proof of Theorem \ref{thm_main}}

\label{sec_proof}

Let $X,Y\in L^\infty(M, TM)$ be real vector fields  and let $q\in L^\infty(M, \R)$.  Let $u\in H^1(M^0)$ be a solution to 
\begin{equation}
\label{eq_int_4_advec_poten}
(-\Delta + X +\div(Y) + q) u=0 \quad \text{in}\quad \mathcal{D}'(M^0).   
\end{equation}
We shall define the trace of the advection normal derivative $(\p_{\nu}-\langle Y, \nu\rangle )u\in H^{-1/2}(\p M)$ as follows. Let $\psi\in H^{1/2}(\p M)$. Then letting $v\in H^1(M^0)$ be a continuous extension of $\psi$, we set 
\begin{equation}
\label{eq_int_5_advec_poten}
\begin{aligned}
\langle& (\p_\nu-\langle Y, \nu\rangle ) u, \psi \rangle_{H^{-1/2}(\p M)\times H^{1/2}(\p M)}\\
&=\int_M\langle \nabla u, \nabla v\rangle dV +\int_M X(u) vdV
-\int_{M} Y(uv)dV +\int_{M} q uvdV.
\end{aligned}
\end{equation}
As $u$ satisfies \eqref{eq_int_4_advec_poten}, the above definition of the trace $(\p_\nu-\langle Y, \nu\rangle)u\in H^{-1/2}(\p M)$  is independent of the choice of an extension of $\psi$.

Consider $L_X=-\Delta+X$. Associated to $L_X$ is the formal $L^2$ adjoint $L_X^*$ defined by
\[
(L_X u, v)_{L^2(M)}=\langle u, L_X^*\overline{v}\rangle_{H^1_0(M^0), H^{-1}(M^0)},\quad u,v\in C^\infty_0(M^0)
\]
so that 
\[
L_X^*=-\Delta-X-\div(X): C_0^\infty(M^0)\to H^{-1}(M^0). 
\] 
In particular, when  $u\in H^1(M^0)$ is a solution to
\begin{equation}
\label{eq_int_4}
L^*_{X}u=0 \quad \text{in}\quad \mathcal{D}'(M^0),
\end{equation}
using \eqref{eq_int_5_advec_poten}, we get 
\begin{equation}
\label{eq_int_5}
\langle (\p_\nu+\langle X, \nu\rangle ) u, \varphi \rangle_{H^{-1/2}(\p M)\times H^{1/2}(\p M)}=\int_M\langle \nabla u, \nabla v\rangle dV +\int_M u X(v)dV.
\end{equation}

Our starting point is the following integral identity.  
\begin{prop}
\label{prop_integral_identity}
Let $X^{(1)}, X^{(2)}\in L^\infty(M,TM)$ be real  vector fields. If $\Lambda_{X^{(1)}}=\Lambda_{X^{(2)}}$ then 
\begin{equation}
\label{eq_integral_0}
\int_M (X^{(1)}-X^{(2)})(u_1) u_2dV=0,
\end{equation}
for any $u_1, u_2\in H^1(M^0)$ satisfying  $(-\Delta+X^{(1)})u_1=0$ and $(-\Delta -X^{(2)}-\div (X^{(2)}))u_2=0$  in $\mathcal{D}'(M^0)$. 
\end{prop}

\begin{proof}
As $\Lambda_{X^{(1)}}=\Lambda_{X^{(2)}}$, there is a solution $v_2\in H^1(M^0)$ to  
\begin{equation}
\label{eq_integral_1_-1}
(-\Delta+X^{(2)})v_2=0\quad \text{in}\quad  \mathcal{D}'(M^0)
\end{equation}
 such that 
\begin{equation}
\label{eq_integral_1_0}
u_1|_{\p M}=v_2|_{\p M}, \quad \p_{\nu }u_1|_{\p M}=\p_{\nu}v_2|_{\p M}. 
\end{equation}
Hence, \eqref{eq_integral_1_0} implies that 
\begin{equation}
\label{eq_integral_1}
\langle \p_\nu u_1, u_2\rangle_{H^{-1/2}(\p M)\times H^{1/2}(\p M)}=\langle \p_\nu v_2, u_2\rangle_{H^{-1/2}(\p M)\times H^{1/2}(\p M)}
\end{equation}
Using \eqref{eq_integral_1_-1}, the fact that $(-\Delta_g -X^{(2)}-\div (X^{(2)}))u_2=0$ and \eqref{eq_int_5},  we get 
\begin{equation}
\label{eq_integral_2}
\begin{aligned}
\langle \p_\nu v_2, u_2\rangle_{H^{-1/2}(\p M)\times H^{1/2}(\p M)}&=\int_M \langle \nabla v_2, \nabla u_2\rangle dV + \int_{M} X^{(2)}(v_2) u_2 dV\\
&=\langle (\p_\nu+ \langle X^{(2)}, \nu\rangle   ) u_2, v_2\rangle_{H^{-1/2}(\p M)\times H^{1/2}(\p M)}\\
&=\langle (\p_\nu+\langle X^{(2)}, \nu\rangle ) u_2, u_1\rangle_{H^{-1/2}(\p M)\times H^{1/2}(\p M)}\\
&=\int_M \langle \nabla u_2, \nabla u_1\rangle dV + \int_{M} X^{(2)}(u_1) u_2 dV.
\end{aligned}
\end{equation}
On the other hand, using the equation for $u_1$, we obtain that 
\begin{equation}
\label{eq_integral_3}
\langle \p_\nu u_1, u_2\rangle_{H^{-1/2}(\p M)\times H^{1/2}(\p M)}=\int_M \langle \nabla u_1, \nabla u_2\rangle dV + \int_{M} X^{(1)}(u_1) u_2 dV.
\end{equation}
The claim follows from \eqref{eq_integral_1}, \eqref{eq_integral_2} and \eqref{eq_integral_3}. 
\end{proof}

We shall need the following result the proof of which is similar to \cite[Proposition 3.4]{Krup_Uhlmann_magnet}.
\begin{lem}
\label{lem_Dirichlet_on_bigger}
Let $(M,g)$ and $(\tilde M,g)$ be smooth compact Riemannian manifolds with smooth boundaries such that $M\subset \tilde M^0$, and let   $X^{(j)}\in (H^{s}\cap L^\infty)(\tilde M^0,T\tilde M^0)$, $0\le s\le 1$, $j=1,2$, and assume that 
\[
X^{(1)}=X^{(2)}\quad \text{in}\quad \tilde M\setminus M.
\]
If $\Lambda_{X^{(1)}}=\Lambda_{X^{(2)}}$ then $\Lambda'_{X^{(1)}}=\Lambda'_{X^{(2)}}$, where $\Lambda'_{X^{(j)}}$ is the Dirichlet--to--Neumann map for the operator $L_{X^{(j)}}$ on $\tilde M$, $j=1,2$.  
\end{lem}

Let $X^{(1)}, X^{(2)}\in (H^{s}\cap L^\infty)(M^0,TM^0)$, $0\le s\le 1$. When $s=1$, an application of  Theorem \ref{prop_boundary_rec} gives that 
$X^{(1)}|_{\p M}=X^{(2)}|_{\p M}$ in $H^{1/2}(\p M, TM|_{\p M})$. We know from the beginning of Section \ref{sec_CGO} that  we can extend  $X^{(1)}$, $X^{(2)}$ to compactly supported vector fields in $(H^1\cap L^\infty)(\R\times M_0^0, T(\R\times M^0_0))$ and arguing as in \cite[Section 5.1]{Krup_Uhlmann_calderon}, we may modify $X^{(2)}$ so that the extensions of $X^{(1)}$ and $X^{(2)}$ agree on $\R\times M_0^0\setminus M$, and their regularity is preserved. In view of the future applications to the proof of Theorem \ref{thm_main_2} it will be convenient to continue working with a general $0\le s\le 1$ and to that end, when $0\le s<1$, we assume that we can extend  $X^{(1)}$ and $X^{(2)}$ to compactly supported vector fields in $(H^s\cap L^\infty)(\R\times M_0^0), T(\R\times M^0_0))$ so that $X^{(1)}=X^{(2)}$ outside of $M$.  Applying Lemma \ref{lem_Dirichlet_on_bigger}, we may and will assume therefore that the admissible manifold $(M,g)$ is simply connected with connected boundary and $X^{(1)}$ and $X^{(2)}$ are compactly supported in the interior of $M$. 

As $\Lambda_{X^{(1)}}=\Lambda_{X^{(2)}}$, it follows from Proposition \ref{prop_integral_identity} that 
\begin{equation}
\label{eq_integral_0_proof}
\int_M (X^{(1)}-X^{(2)})(u_1) u_2dV=0,
\end{equation}
for any $u_1, u_2\in H^1(M^0)$ satisfying  
\begin{equation}
\label{sec_200_1}
(-\Delta+X^{(1)})u_1=0
\end{equation}
 and 
 \begin{equation}
\label{sec_200_2}
 (-\Delta -X^{(2)}-\div (X^{(2)}))u_2=0.
 \end{equation}

By Proposition \ref{prop_cgo-admiss} for all $h>0$ small enough, there are  solutions $u_1,u_2\in H^1(M^0)$ to  \eqref{sec_200_1} and \eqref{sec_200_2}, respectively, of the form
\begin{equation}
\label{eq_proof_2}
\begin{aligned}
u_1&=e^{-\frac{\rho}{h}}(|g|^{-1/4} c^{1/2} e^{\Phi^{(1)}_{h}} a_0(x_1,r)b(\theta) +r_1 )=e^{-\frac{\rho}{h}} (\alpha_1+r_1),\\
u_2&=e^{\frac{\rho}{h}}(|g|^{-1/4} c^{1/2} e^{\Phi^{(2)}_{h}}  +r_2)=e^{\frac{\rho}{h}}(\alpha_2  +r_2 ),
\end{aligned}
\end{equation}
respectively. Here $\rho=x_1+ir $, 
\begin{equation}
\label{eq_proof_3}
\alpha_1=|g|^{-1/4} c^{1/2} e^{\Phi^{(1)}_{h}} a_0(x_1,r)b(\theta), \quad \alpha_2=|g|^{-1/4} c^{1/2} e^{\Phi^{(2)}_{h}},
\end{equation}
and
\begin{equation}
\label{eq_proof_4}
\|r_j\|_{L^2(M^0)}=o(h^{\frac{s}{2}}), \quad \|dr_j\|_{L^2(M^0)}=o(h^{\frac{s}{2}-1}), \quad h\to 0, \quad j=1,2.
\end{equation}
Furthermore, $\Phi^{(j)}_{h}\in C^\infty (M) $ satisfies  \eqref{eq_cgo_19} and \eqref{eq_cgo_19_L_infty}, and 
\begin{equation}
\label{eq_proof_4_0}
\|\Phi^{(j)}-\Phi^{(j)}_{h}\|_{L^2(M)}=o(h^{\frac{s}{2}}), \quad h\to 0, 
\end{equation}
where 
\begin{align*}
&\Phi^{(1)}(x_1,r,\theta) = \frac{1}{4} \frac{1}{\pi (x_1+i r)} * (X^{(1)}_1 +i X^{(1)}_r)c,\\
&\Phi^{(2)}(x_1,r,\theta)= -\frac{1}{4} \frac{1}{\pi (x_1+i r)} * (X^{(2)}_1 +i X^{(2)}_r)c,
\end{align*}
with $X^{(j)}=X^{(j)}_1 \p_{x_1}+X^{(j)}_r\p_r+X^{(j)}_\theta \p_\theta$, $j=1,2$.  Thus, 
\begin{equation}
\label{eq_proof_4_1}
\Phi=\Phi^{(1)}+\Phi^{(2)}
\end{equation}
 satisfies 
\begin{equation}
\label{eq_proof_4_2}
\overline{\p} \Phi= \frac{1}{4} (\tilde X_1 +i \tilde X_r)c, \quad \text{in}\quad M,
\end{equation}
where $\tilde X=X^{(1)}-X^{(2)}$. Furthermore, $\alpha_1$ and $\alpha_2$ satisfy \eqref{eq_cgo_19_ampl_estimates}.

We shall next insert  $u_1$ and $u_2$, given by \eqref{eq_proof_2},  into \eqref{eq_integral_0_proof}, multiply it by $h$, and let $h\to 0$.  To that end, we first compute
\begin{equation}
\label{eq_proof_6}
\tilde X(u_1)u_2=-\frac{1}{h}\tilde X(\rho) (\alpha_1\alpha_2 +\alpha_1r_2+\alpha_2 r_1+r_1r_2)+ (\tilde X(\alpha_1)+\tilde X(r_1))(\alpha_2+r_2).
\end{equation}

Using \eqref{eq_cgo_19_ampl_estimates} and \eqref{eq_proof_4}, we get 
\begin{equation}
\label{eq_proof_6_01}
\begin{aligned}
\bigg|  \int_M \tilde X(\rho)& (\alpha_1r_2+\alpha_2 r_1+r_1r_2)dV  \bigg|\\
&\le \mathcal{O}(1)(\|\alpha_1 \|_{L^\infty}\|r_2\|_{L^2}+ \|\alpha_2 \|_{L^\infty}\|r_1\|_{L^2}+ \|r_1 \|_{L^2}\|r_2\|_{L^2})=o(h^{\frac{s}{2}}),
\end{aligned}
\end{equation}
and 
\begin{equation}
\label{eq_proof_6_02}
\begin{aligned}
\bigg| h\int_M (\tilde X(\alpha_1)& +\tilde X(r_1))(\alpha_2+r_2) dV\bigg|\\
&\le \mathcal{O}(h)(\|\nabla \alpha_1\|_{L^2}+\|\nabla r_1\|_{L^2})(\|\alpha_2\|_{L^\infty}+\|r_2\|_{L^2})=o(h^{\frac{s}{2}}).
\end{aligned}
\end{equation}

In view of \eqref{eq_proof_6}, \eqref{eq_proof_6_01}, and \eqref{eq_proof_6_02},    we obtain from \eqref{eq_integral_0_proof}  that 
\begin{equation}
\label{eq_proof_7}
\lim_{h\to 0} \int_M \tilde X(\rho) \alpha_1\alpha_2dV=\lim_{h\to 0}\int_M \tilde X(\rho) |g|^{-1/2} c e^{\Phi^{(1)}_h+\Phi^{(2)}_h}a_0bdV=0.
\end{equation}

Let us show that 
\begin{equation}
\label{eq_proof_8}
\lim_{h\to 0}\int_M \tilde X(\rho) |g|^{-1/2} c e^{\Phi^{(1)}_h+\Phi^{(2)}_h}a_0bdV=
\int_M \tilde X(\rho) |g|^{-1/2} c e^{\Phi}a_0bdV,
\end{equation}
where $\Phi$ is given by \eqref{eq_proof_4_1}. To that end, we get
\begin{align*}
\bigg|  \int_M \tilde X(\rho) |g|^{-1/2} c \big(e^{\Phi^{(1)}_h+\Phi^{(2)}_h}- e^{\Phi} \big) a_0bdV  \bigg|\le C\| e^{(\Phi^{(1)}_{h}+\Phi^{(2)}_{h})}-e^{\Phi} \|_{L^2(M)}\\
\le \| \Phi^{(1)}_{h}+\Phi^{(2)}_{h}-\Phi_1-\Phi_2 \|_{L^2(M)}=o(h^{\frac{s}{2}})\to 0, \quad h\to 0,
\end{align*}
proving the claim. 
Here we have used \eqref{eq_proof_4_0}, the inequality 
\[
|e^z-e^w|\le |z-w|e^{\max(\text{Re}\, z,\text{Re}\, w)}, \quad z,w\in \C,
\]
and the fact that $\Phi^{(j)}, \Phi^{(j)}_{h}\in L^\infty(M)$ and $\|\Phi^{(j)}_{h}\|_{L^\infty(M)}\le C$ uniformly in $h$.

Hence, it follows from \eqref{eq_proof_7} and \eqref{eq_proof_8} that 
\[
\int_M \tilde X(\rho) |g|^{-1/2} c e^{\Phi}a_0bdV=0,
\]
and thus, writing this integral in the global coordinates $(x_1,r,\theta)$ and using that $dV_g=|g|^{1/2}dx_1drd\theta$, we get
\begin{equation}
\label{eq_proof_cov_1}
\int_M (\tilde X_1+i \tilde X_r) c  e^{\Phi}a_0(x_1,r)b(\theta)dx_1dr d\theta=0.
\end{equation}

Using only that $X^{(1)}, X^{(2)}\in L^\infty(M, TM)$ and arguing as in the proof of Theorem 1.3 in \cite{Krup_Uhlmann_magnet_manifolds}, see also \cite{DKSU_2007},  we obtain from \eqref{eq_proof_cov_1} that 
\begin{equation}
\label{eq_proof_12}
\int\!\!\!\int\!\!\!\int_{\R\times D} (\tilde X_1+i \tilde X_r) c  e^{i\lambda(x_1+ir)} b(\theta) dx_1drd\theta=0, \quad \lambda\in \R.
\end{equation}

Letting $\tilde X^\flat=g_{jk}\tilde X^{j}dx^k\in (L^\infty\cap H^s)(M^0, T^*M^0)$, $0\le s\le 1$, be the one form, corresponding to the vector field $\tilde X$,   we have $\tilde X_1^\flat=g_{j1}\tilde X^1=c\tilde X^1$, and $\tilde X_r^\flat=c\tilde X^r$, and therefore, \eqref{eq_proof_12} can be rewritten as follows,
\begin{equation}
\label{eq_proof_13}
\int\!\!\!\int\!\!\!\int_{\R\times D} (\tilde X^\flat_1+i\tilde X^\flat_r) e^{i\lambda(x_1+ir)} b(\theta) dx_1drd\theta=0.
\end{equation}

Now the argument in the proof of Theorem 1.3 in \cite{Krup_Uhlmann_magnet_manifolds} applies as it stands allowing us to  conclude that  $d \tilde X^\flat =0$ in $M$ for all $0\le s\le 1$.

We shall next show that $\tilde X=0$ in $M$, when $s=1$. Since $M$ is simply connected, by the Poincar\'e lemma for currents, see \cite{de_Rham},  we conclude that  there is $\phi\in \mathcal{D}'(M)$ such that 
\begin{equation}
\label{eq_sec_10_1}
d\phi = \tilde X^\flat  \in (L^\infty\cap H^s)(M^0, T^*M^0) \cap \mathcal{E}'(M^0, T^*M^0), \quad 0\le s\le 1. 
\end{equation}
It follows from \cite[Theorem 4.5.11]{Hormander_book_1} that $\phi$ is continuous and $\phi$ is a constant $c$ near $\p M$.  Therefore, $\phi\in W^{1,\infty}(M^0)$, and since the boundary $\p M$ is connected by considering $\phi-c$, we may assume that $\phi=0$ on $\p M$. 
It follows from \eqref{eq_sec_10_1} that 
\begin{equation}
\label{eq_sec_10_1_gradient}
\nabla \phi=\tilde X= X^{(1)}-X^{(2)}.
\end{equation}

We have
\begin{align*}
e^{\frac{\phi}{2}}\circ (-\Delta +X^{(2)})\circ e^{-\frac{\phi}{2}}=&-\Delta +X^{(2)}+ \nabla \phi
+
\big(\frac{1}{2}\Delta\phi  -\frac{1}{4} \langle d\phi, d\phi \rangle-\frac{1}{2} X^{(2)}(\phi)\big). 
\end{align*}
Thus, in view of \eqref{eq_sec_10_1_gradient}, we get 
\begin{equation}
\label{eq_sec_200_3}
P_{X^{(1)},Y, q}:=e^{\frac{\phi}{2}}\circ (-\Delta +X^{(2)})\circ e^{-\frac{\phi}{2}}=-\Delta +X^{(1)}+ \div Y+ q,
\end{equation}
where 
\begin{equation}
\label{eq_sec_200_3_poten}
\begin{aligned}
Y=
\frac{1}{2}\nabla\phi \in (L^\infty\cap H^s)(M^0, TM^0), \quad q=  -\frac{1}{4} \langle d\phi, d\phi \rangle-\frac{1}{2} X^{(2)}(\phi)\in L^\infty(M).
\end{aligned}
\end{equation}

Using \eqref{eq_int_5_advec_poten}, we define the set of the Cauchy data for $P_{X^{(1)},Y, q}$,
\[
C_{X^{(1)},Y, q}=\{(u|_{\p M}, (\p_\nu-\langle Y, \nu\rangle ) u|_{\p M}): u\in H^1(M^0), P_{X^{(1)},Y,q}u=0 \text{ in }\mathcal{D}'(M^0)\}. 
\]

Introducing the set of the Cauchy data for $L_{X^{(j)}}$, $j=1,2$, as follows
\[
C_{X^{(j)}}=\{(u|_{\p M}, \Lambda_{X^{(j)}}(u|_{\p M})): u\in H^1(M^0), L_{X^{(j)}}u=0 \text{ in }\mathcal{D}'(M^0)\},
\]
we shall show that 
\begin{equation}
\label{eq_sec_200_4}
C_{X^{(1)},Y, q}=C_{X^{(2)}}.
\end{equation}
In order to see \eqref{eq_sec_200_4}, let $u\in H^1(M^0)$ be a solution to  $L_{X^{(2)}}u=0$ in $\mathcal{D}'(M^0)$. Then it follows from \eqref{eq_sec_200_3} that $e^{\frac{\phi}{2}}u\in H^1(M^0)$ satisfies $P_{X^{(1)},Y,q}(e^{\frac{\phi}{2}}u)=0$ in $\mathcal{D}'(M^0)$. We have 
$e^{\frac{\phi}{2}}u|_{\p M}=u|_{\p M}$. Using \eqref{eq_int_5_advec_poten}, \eqref{eq_sec_10_1_gradient}, and \eqref{eq_int_2_2}, the fact that $\nabla \phi (u)=\langle \nabla \phi, \nabla u\rangle$,  for $\psi\in H^{1/2}(\p M)$, we get
\begin{align*}
\langle (\p_\nu-\langle Y, \nu\rangle)& (e^{\frac{\phi}{2}}u), \psi \rangle_{H^{-1/2}(\p M)\times H^{1/2}(\p M)}\\
&=\langle (\p_\nu-\langle Y, \nu\rangle ) (e^{\frac{\phi}{2}}u), e^{-\frac{\phi}{2}} \psi \rangle_{H^{-1/2}(\p M)\times H^{1/2}(\p M)}\\
&=\int_M\langle \nabla (e^{\frac{\phi}{2}} u), \nabla (e^{-\frac{\phi}{2}} v)\rangle dV +\int_M (X^{(2)}+\nabla\phi)(e^{\frac{\phi}{2}}  u) e^{-\frac{\phi}{2}} vdV\\
&-\int_{M} \frac{1}{2}\nabla\phi (uv)dV +\int_{M} (-\frac{1}{4} \langle \nabla \phi, \nabla\phi \rangle-\frac{1}{2} X^{(2)}(\phi)) uvdV\\
&=\int_{M} \langle \nabla u, \nabla v\rangle dV  +\int_M X^{(2)}(u)v dV=\langle \p_\nu u, \psi \rangle_{H^{-1/2}(\p M)\times H^{1/2}(\p M)},
\end{align*}
showing \eqref{eq_sec_200_4}. 

The fact that $\Lambda_{X^{(1)}}=\Lambda_{X^{(2)}}$ and \eqref{eq_sec_200_4} imply that 
\begin{equation}
\label{eq_sec_200_5}
C_{X^{(1)},Y, q}=C_{X^{(1)}}.
\end{equation}
 Let us show, as a consequence of \eqref{eq_sec_200_5},  that we have the following integral identity,
\begin{equation}
\label{eq_integral_identity_new_pot}
\int_M (qu_1 u_2 -Y(u_1u_2)) dV=0,
\end{equation}
for any $u_1,u_2\in H^1(M^0)$ satisfying 
\begin{equation}
\label{eq_sec_200_6_-2}
L_{X^{(1)}}u_1=0\quad \text{in}\quad \mathcal{D}'(M^0),
\end{equation}
and 
\begin{equation}
\label{eq_sec_200_6_-1}
(-\Delta -X^{(1)} -\div (X^{(1)}) +\div(Y)+q)u_2=0\quad \text{in}\quad \mathcal{D}'(M^0).
\end{equation}
Indeed,  \eqref{eq_sec_200_5} implies that there is a solution $v\in H^1(M^0)$ to \eqref{eq_int_4_advec_poten} such that 
\begin{equation}
\label{eq_sec_200_6}
u_1|_{\p M}=v|_{\p M}, \quad \p_\nu u_1|_{\p M}=(\p_\nu-\langle Y, \nu\rangle)v|_{\p M}.
\end{equation}
Thus, \eqref{eq_sec_200_6} implies that 
\begin{equation}
\label{eq_sec_200_7}
\langle \p_\nu u_1, u_2\rangle_{H^{-1/2}(\p M)\times H^{1/2}(\p M)}=\langle (\p_\nu-\langle Y, \nu\rangle)v|_{\p M}, u_2\rangle_{H^{-1/2}(\p M)\times H^{1/2}(\p M)}.
\end{equation}
As $v$, $u_2$ solve  \eqref{eq_int_4_advec_poten} and \eqref{eq_sec_200_6_-1}, correspondingly,  by   \eqref{eq_int_5_advec_poten}, we obtain that 
\begin{equation}
\label{eq_sec_200_8}
\begin{aligned}
\langle& (\p_\nu-\langle Y, \nu\rangle ) v, u_2 \rangle_{H^{-1/2}(\p M)\times H^{1/2}(\p M)}\\
&=\int_M\langle \nabla v, \nabla u_2\rangle dV +\int_M X^{(1)}(v) u_2dV
-\int_{M} Y(vu_2)dV +\int_{M} q vu_2dV\\
&=\langle (\p_\nu - \langle X^{(1)}, \nu\rangle +\langle Y, \nu\rangle ) u_2, v \rangle_{H^{-1/2}(\p M)\times H^{1/2}(\p M)}\\
&=\langle (\p_\nu - \langle X^{(1)}, \nu\rangle +\langle Y, \nu\rangle ) u_2, u_1 \rangle_{H^{-1/2}(\p M)\times H^{1/2}(\p M)}\\
&=\int_M\langle \nabla u_2, \nabla u_1\rangle dV +\int_M u_2 X^{(1)}(u_1) dV
-\int_{M} Y(u_1u_2 )dV +\int_{M} q u_1u_2dV.
\end{aligned}
\end{equation}
On the other hand, as $u_1$ solves \eqref{eq_sec_200_6_-2}, we get
\begin{equation}
\label{eq_sec_200_9}
\langle \p_\nu u_1, u_2\rangle_{H^{-1/2}(\p M)\times H^{1/2}(\p M)}= \int_M\langle \nabla u_1, \nabla u_2\rangle dV +\int_M X^{(1)}(u_1)u_2 dV.
\end{equation}
Hence, \eqref{eq_integral_identity_new_pot} follows from \eqref{eq_sec_200_7},  \eqref{eq_sec_200_8}, and \eqref{eq_sec_200_9}.

Now by Proposition \ref{prop_cgo-admiss},  for all $h>0$ small enough, there are  solutions $u_1,u_2\in H^1(M^0)$ of the form
\begin{equation}
\label{eq_proof_2_pot}
\begin{aligned}
u_1&=e^{-\frac{\rho}{h}}(|g|^{-1/4} c^{1/2} e^{i\lambda \rho} e^{\Phi^{(1)}_{h}} b(\theta) +r_1 )=e^{-\frac{\rho}{h}} (\alpha_1+r_1),\\
u_2&=e^{\frac{\rho}{h}}(|g|^{-1/4} c^{1/2} e^{\Phi^{(2)}_{h}}  +r_2)=e^{\frac{\rho}{h}}(\alpha_2  +r_2 ),
\end{aligned}
\end{equation}
to the equations \eqref{eq_sec_200_6_-2} and \eqref{eq_sec_200_6_-1}, respectively. 
Here $\rho=x_1+ir $, 
\begin{equation}
\label{eq_proof_3_pot}
\alpha_1=|g|^{-1/4} c^{1/2} e^{i\lambda \rho}  e^{\Phi^{(1)}_{h}} b(\theta), \quad \alpha_2=|g|^{-1/4} c^{1/2} e^{\Phi^{(2)}_{h}},
\end{equation}
$\lambda$ is a fixed real number,  $b\in C^\infty(\mathbb{S}^{n-1})$ is a fixed function, and
\begin{equation}
\label{eq_proof_4_pot}
\|r_j\|_{L^2(M^0)}=o(h^{\frac{s}{2}}), \quad \|dr_j\|_{L^2(M^0)}=o(h^{\frac{s}{2}-1}), \quad h\to 0, \quad j=1,2.
\end{equation}
Furthermore, it follows from \eqref{eq_cgo_11} that $\Phi^{(j)}_{h}\in C^\infty (M) $ are given by 
\begin{align*}
&\Phi^{(1)}_h(x_1,r,\theta) = \frac{1}{4} \frac{1}{\pi (x_1+i r)} * ((X^{(1)}_h)_1 +i (X^{(1)}_h)_r) c,\\
&\Phi^{(2)}_h(x_1,r,\theta) = -\frac{1}{4} \frac{1}{\pi (x_1+i r)} * ((X^{(1)}_h)_1 +i (X^{(1)}_h)_r) c,
\end{align*}
and therefore 
\begin{equation}
\label{eq_proof_3_pot_phys}
\Phi^{(1)}_h+\Phi^{(2)}_h=0.
\end{equation}  
Finally, the amplitudes $\alpha_1$ and $\alpha_2$ satisfy \eqref{eq_cgo_19_ampl_estimates}.

Next let us show that inserting  $u_1$ and $u_2$, given by \eqref{eq_proof_2_pot}  into the integral identify \eqref{eq_integral_identity_new_pot} and letting $h\to 0$, we get
\begin{equation}
\label{eq_sec_200_10}
\lim_{h\to 0}\int_M\bigg(q\alpha_1\alpha_2-\frac{1}{2}\langle \nabla\phi, \nabla(\alpha_1\alpha_2) \rangle\bigg) dV=0.
\end{equation} 
Indeed, using \eqref{eq_proof_4_pot} and \eqref{eq_cgo_19_ampl_estimates}, we see that  
\begin{equation}
\label{eq_sec_9_2}
\begin{aligned}
\bigg|  \int_{M}  q(\alpha_1r_2+\alpha_2 r_1+r_1r_2)dV\bigg|\le 
C\| q\|_{L^\infty(M)} \big( \|\alpha_1\|_{L^\infty(M)}\|r_2\|_{L^2(M)}\\
+ \|\alpha_2\|_{L^\infty(M)}\|r_1\|_{L^2(M)}  +\|r_1\|_{L^2(M)}\|r_2\|_{L^2(M)} \big)=o(h^{\frac{s}{2}}).
\end{aligned}
\end{equation}

As $\nabla \phi\in (H^s\cap L^\infty)(M^0, TM^0)\cap\mathcal{E}'(M^0, TM^0)$, $0\le s\le 1$, an application of Proposition \ref{prop_app_R_n} shows that there is a family of $\phi_\tau\in C^\infty_0(M^0)$, $\tau>0$, such that 
\begin{equation}
\label{eq_sec_200_11}
\|\nabla \phi-\nabla\phi_\tau\|_{L^2}=o(\tau^s),
\end{equation}
\begin{equation}
\label{eq_sec_200_12}
\| \nabla \phi_\tau\|_{L^\infty}=\mathcal{O}(1), \quad \| \nabla^2 \phi_\tau \|_{L^2}=\begin{cases} o(\tau^{s-1}), & 0\le s<1,\\
\mathcal{O}(1), & s=1.
\end{cases}
\end{equation}

Using \eqref{eq_cgo_19_ampl_estimates}, \eqref{eq_proof_4_pot}, \eqref{eq_sec_200_11} and \eqref{eq_sec_200_12}, we get 
\begin{equation}
\label{eq_sec_9_3}
\begin{aligned}
&\bigg|  \int_{M} \langle \nabla\phi, \nabla (\alpha_1r_2) \rangle dV\bigg|\le   \int_{M} \big| \langle \nabla \phi-\nabla \phi_\tau, \nabla ( \alpha_1r_2)\rangle \big| dV+ 
\int_{M} \big| \langle \Delta\phi_\tau,  \alpha_1r_2\rangle  \big| dV\\
&\le \| \nabla \phi -\nabla \phi_\tau\|_{L^2(M)} \big(  \|\nabla \alpha_1\|_{L^\infty(M)} \|r_2\|_{L^2(M)} + \|\alpha_1\|_{L^\infty(M)} \|\nabla  r_2\|_{L^2(M)} \big)\\
&+ \|\Delta \phi_\tau\|_{L^2(M)}\| \alpha_1\|_{L^\infty(M)}\|r_2\|_{L^2(M)}=o(\tau^s) o(h^{\frac{s}{2}-1})+\mathcal{O}(\tau^{s-1})o(h^{\frac{s}{2}})=o(h^{\frac{3s}{2}-1}),
\end{aligned}
\end{equation}
provided that $\tau=h$. 
Similarly, we have 
\begin{equation}
\label{eq_sec_9_4}
\bigg|  \int_{M} \langle \nabla\phi, \nabla (\alpha_2r_1) \rangle dV\bigg|=o(h^{\frac{3s}{2}-1}).
\end{equation}

Using \eqref{eq_proof_4_pot}, we have 
\begin{equation}
\label{eq_sec_9_5}
\begin{aligned}
\bigg|  \int_{M} \langle \nabla\phi, \nabla (r_1r_2) \rangle dV\bigg|\le \| \nabla \phi\|_{L^\infty(M)} \big(  \|\nabla r_1\|_{L^2(M)} \|r_2\|_{L^2(M)} \\
+ \|r_1\|_{L^2(M)} \|\nabla  r_2\|_{L^2(M)} \big)=o(h^{s-1}).
\end{aligned}
\end{equation}
Thus, when $s=1$, \eqref{eq_sec_200_10} follows from \eqref{eq_sec_9_2}, \eqref{eq_sec_9_3}, \eqref{eq_sec_9_4} and \eqref{eq_sec_9_5}.

We are now able to complete the proof of Theorem \ref{thm_main}. Here  $X^{(1)},X^{(2)}\in (H^1\cap L^\infty)(M^0, TM^0)$ so that  $\nabla\phi\in H^1_0(M)$. Therefore,   integrating by parts in \eqref{eq_sec_200_10}, and using \eqref{eq_proof_3_pot}, \eqref{eq_proof_3_pot_phys}, we get 
\begin{equation}
\label{eq_sec_200_14}
0=\lim_{h\to 0}\int_M \bigg(q+\frac{1}{2}\Delta \phi\bigg) \alpha_1\alpha_2 dV= \int_M \bigg(q+\frac{1}{2}\Delta \phi\bigg) |g|^{-1/2} e^{i\lambda(x_1+ir)} c bdV.
\end{equation}
Writing  
\begin{equation}
\label{eq_sec_200_14_1}
\tilde q= q+\frac{1}{2}\Delta \phi\in L^2(M)
\end{equation}
  and $dV= |g|^{1/2} dx_1dr d\theta$,  we have from \eqref{eq_sec_200_14},
\begin{equation}
\label{eq_sec_200_15}
\int \!\!\int \!\!\int_{\R\times D}  \tilde q(x_1,r,\theta)  c(x_1,r,\theta)  e^{i\lambda (x_1+ir)}  b(\theta)   dx_1 dr d\theta=0.
\end{equation}
Here we take $\omega\in \p D$ to define the Riemannian polar normal coordinates $(r, \theta)$. Setting 
\[
f(r,\theta)=\int_{-\infty}^\infty  e^{i\lambda x_1} \tilde q(x_1,r,\theta)  c(x_1,r,\theta)  dx_1,
\]
we have $f\in L^1(D)$. Thus, it follows from  \eqref{eq_sec_200_15} that 
\begin{equation}
\label{eq_sec_200_16}
 \int_{\mathbb{S}^{n-2}} \int_0^{\tau(\omega, \theta)}   f(r,\theta) e^{-\lambda r}  b(\theta)   dr d\theta=0,
\end{equation}
for all $\omega\in \p D$, all $b\in C^\infty(\mathbb{S}^{n-2})$. 

The integral in \eqref{eq_sec_200_16} is related to the attenuated geodesic ray transform acting on the function $f$ in $D$ with constant attenuation $-\lambda$. In order to proceed we shall need the following result from  \cite[Lemma 5.1]{DKSalo_2013} for $f\in L^1(D)$.

\begin{prop}
\label{prop_Kenig}
Let $(D,g_0)$ be an $(n-1)$-dimensional simple manifold and let $f\in L^1(D)$.  Consider the integrals
\[
\int_{\mathbb{S}^{n-2}}\int_0^{\tau(\omega, \theta)}  f(\gamma_\theta(r)) e^{-\lambda r} b(\theta)dr d\theta,
\] 
where $(r,\theta)$ are polar normal coordinates in $(D,g_0)$ centered at some $\omega\in \p D$ and $\tau(\omega, \theta)$ is the time when the geodesic $\gamma_\theta:  r\mapsto (r, \theta)$ exits $D$.  If  $|\lambda|$ is sufficiently small, and if these integrals vanish  for all $\omega\in \p D$ and all $b\in C^\infty(\mathbb{S}^{n-2})$, then $f=0$. 
\end{prop}

Now varying the point $\omega\in \p D$ in the construction of the complex geometric optics solution in Proposition \ref{prop_cgo-admiss} and applying  Proposition \ref{prop_Kenig} to  \eqref{eq_sec_200_16},  we get 
\[
f(r,\theta)=\int_{-\infty}^\infty  e^{i\lambda x_1} \tilde q(x_1,r,\theta)  c(x_1,r,\theta)  dx_1=0,
\]
for a.a. $(r,\theta)$ and for all $|\lambda|$ sufficiently small, and therefore, for all $\lambda\in \C$ by analyticity of the Fourier transform of the compactly supported function $\tilde q(\cdot,r,\theta)  c(\cdot,r,\theta)\in L^1 $ for a.a. $(r,\theta)$. Thus, in view of \eqref{eq_sec_200_14_1} and  \eqref{eq_sec_200_3_poten}, we conclude that 
\begin{equation}
\label{eq_sec_200_17}
\Delta \phi -\frac{1}{2}\langle \nabla \phi, \nabla \phi\rangle -X^{(2)}(\phi)=0. 
\end{equation}
Now letting $ V=X^{(2)}+\frac{1}{2} \langle \nabla \phi, \nabla \rangle\in L^\infty$, we rewrite \eqref{eq_sec_200_17} as the following elliptic boundary value problem,
\begin{equation}
\label{eq_sec_200_18}
\begin{aligned}
\Delta \phi - V(\phi)&=0,\\
\phi|_{\p M}&=0.
\end{aligned}
\end{equation}
Applying the maximum principle of \cite[Chapter 3, Section 8.2]{Aubin_book} to \eqref{eq_sec_200_18}, we get
that $\phi=0$ and therefore, in view of \eqref{eq_sec_10_1_gradient}, $X^{(1)}=X^{(2)}$.  The proof of  Theorem \ref{thm_main} is therefore complete.

\section{Proof of Theorem \ref{thm_main_2}}
\label{sec_proof_thm_2}

Let us now describe modifications in the arguments needed to establish   Theorem \ref{thm_main_2} where $X^{(1)},X^{(2)}\in  H^{\frac{2}{3}} (M^0, TM^0) \cap C^{0, \frac{1}{3}} (M, TM)$. First, arguing as in \cite[Section 2.2]{Krup_Uhlmann_calderon}, we can extend $X^{(1)}$, $X^{(2)}$ to compactly supported vector fields in $ H^{\frac{2}{3}} (\R\times M^0_0, T(\R\times M^0_0)) \cap C^{0, \frac{1}{3}} (\R\times M_0, T(\R\times M_0))$. 
 An application of Corollary \ref{cor_boundary_rec} gives that $X^{(1)}=X^{(2)}$ on $\p M$, and therefore, $(X^{(1)}-X^{(2)})1_{(\R \times M^0_0)\setminus M}\in  H^{\frac{2}{3}} \cap  C^{0, \frac{1}{3}}$, see  \cite[Section 5.1]{Krup_Uhlmann_calderon} and \cite[Theorem 3.4.1, p. 41]{Agranovich_book}.  Replacing $X^{(2)}$ by $X^{(2)}+(X^{(1)}-X^{(2)})1_{(\R \times M^0_0)\setminus M}$, we achieve that $X^{(1)}=X^{(2)}$ outside of $M$.  
 
 From here on, everything works exactly as in the proof of Theorem \ref{thm_main} until we reach the proof of  \eqref{eq_sec_200_10}.   In order to show that  \eqref{eq_sec_200_10} remains valid we observe that \eqref{eq_sec_9_2}, \eqref{eq_sec_9_3}, \eqref{eq_sec_9_4} still hold and we only need to establish an analog of \eqref{eq_sec_9_5}.  
To that end, we shall use that $\nabla\phi= X^{(1)}- X^{(2)}\in C^{0,\frac{1}{3}}(M, TM)$, and therefore, using a partition of unity argument with a regularization in each coordinate patch, in view of Proposition \ref{prop_app_Holder_R_n}, we see  that there is a family $\phi_\tau\in C^\infty_0(M^0)$, $\tau>0$, such that 
\begin{equation}
\label{eq_sec_200_13}
\|\nabla \phi-\nabla\phi_\tau\|_{L^\infty}=\mathcal{O}(\tau^\frac{1}{3}), \
\| \nabla \phi_\tau\|_{L^\infty}=\mathcal{O}(1), \ \| \Delta \phi_\tau \|_{L^\infty}=\mathcal{O}(\tau^{-\frac{2}{3}}),\ \tau\to 0.
\end{equation}
Hence, using   \eqref{eq_proof_4_pot} and \eqref{eq_sec_200_13}, we get 
\begin{align*}
&\bigg|  \int_{M} \langle \nabla\phi, \nabla (r_1r_2) \rangle dV\bigg|\le \int_{M} \big| \langle \nabla \phi-\nabla \phi_\tau, \nabla ( r_1r_2)\rangle \big| dV+ 
\int_{M} \big| \langle \Delta\phi_\tau,  r_1r_2\rangle  \big| dV\\
&\le \| \nabla \phi -\nabla \phi_\tau\|_{L^\infty(M)} \big(  \|\nabla r_1\|_{L^2(M)} \|r_2\|_{L^2(M)} + \|r_1\|_{L^2(M)} \|\nabla  r_2\|_{L^2(M)} \big)\\
&+ \|\Delta \phi_\tau\|_{L^\infty(M)}\| r_1\|_{L^2(M)}\|r_2\|_{L^2(M)}=\mathcal{O}(\tau^{\frac{1}{3}}) o(h^{-\frac{1}{3}})+\mathcal{O}(\tau^{-\frac{2}{3}})o(h^{\frac{2}{3}})=o(1),
\end{align*}
provided that $\tau=h$. Thus,  \eqref{eq_sec_200_10} follows. 

Using  \eqref{eq_sec_200_10} and \eqref{eq_proof_3_pot}, we obtain that 
\begin{equation}
\label{eq_sec_200_23}
\int_{\R\times D} \bigg(q c |g|^{-1/2} e^{i\lambda(x_1+ir)}b-\frac{1}{2}\langle \nabla \phi, \nabla (c |g|^{-1/2} e^{i\lambda(x_1+ir)}b) \rangle \bigg)dV=0.
\end{equation}
Using that $|g|=c^{n}|g_0|$ and  letting 
\[
F=c^{1-n/2}\bigg(q+\frac{1}{2}\Delta\phi\bigg)\in \mathcal{E}'(\R\times D^0),
\]
we rewrite \eqref{eq_sec_200_23} as follows,
\[
\langle F, |g_0|^{-1/2} b e^{i\lambda(x_1+ir)}\rangle_{\R\times D^0}=0,
\]
where $\langle \cdot,\cdot\rangle_{\R\times D^0}$ is the duality between $\mathcal{E}'(\R\times D^0)$ and $C^\infty(\R\times D^0)$. Letting $\hat F_\lambda\in \mathcal{E}'(D^0)$ be the Fourier transform of $F$ with respect to $x_1$, we get 
\begin{equation}
\label{eq_sec_200_24}
\langle \hat  F_{-\lambda}, |g_0|^{-1/2} b e^{-\lambda r}\rangle_{D^0}=0,
\end{equation}
for all $\lambda \in \R$ and $b\in C^\infty(\mathbb{S}^{n-2})$. 
 
In order to proceed we shall need the following result from  \cite[Lemma 4.1]{rodriguez}. 
\begin{prop}
\label{prop_Rodriguez}
Let $(D,g_0)$ be an $(n-1)$-dimensional simple manifold and let $G\in \mathcal{E}'(D^0)$.  Consider the duality pairing, 
\[
\langle  G, |g_0|^{-1/2} b e^{-\lambda r}\rangle_{D^0},
\]
where $(r,\theta)$ are polar normal coordinates in $(D,g_0)$ centered at some $\omega\in \p D$.  If  $|\lambda|$ is sufficiently small, and if these pairings vanish for all $\omega\in \p D$ and all $b\in C^\infty(\mathbb{S}^{n-2})$, then $G=0$. 
\end{prop}
 
An application of Proposition \ref{prop_Rodriguez} to \eqref{eq_sec_200_24}, combined with the analyticity of  $ \hat  F_{-\lambda}$ in $\lambda$ shows that $F=0$.  It follows that $\phi\in H^1(M)$ satisfies \eqref{eq_sec_200_18} and thus, $\phi=0$. The proof of Theorem \ref{thm_main_2} is complete.

\begin{appendix}
\section{Boundary determination of $H^1\cap L^\infty$ advection term}

\label{sec_boundary_rec}

When proving Theorem \ref{thm_main}, an important step consists in determining the boundary values of a vector field $X\in (H^1\cap L^\infty)(M^0, TM^0)$. The purpose of this section is to carry out this step by adapting the method of  \cite{Brown_2001}, \cite{Brown_Salo_2006}, developed in the case of the conductivity equation with possibly discontinuous conductivities, and in the case of the magnetic Schr\"odinger and advection operators with continuous potentials on $\R^n$. Here we shall also rely on a boundary reconstruction result for the magnetic Schr\"odinger operator with continuous potentials on a compact smooth Riemannian manifold with boundary given in our work \cite{Krup_Uhlmann_magnet_manifolds}.

\begin{thm} 
\label{prop_boundary_rec}
Let $(M,g)$ be a smooth compact Riemannian manifold of dimension $n\ge 3$ with smooth boundary $\p M$, and  let $X^{(1)}, X^{(2)}\in (H^1\cap L^\infty)(M^0, TM^0)$.  Assume that $\Lambda_{ X^{(1)}}=\Lambda_{ X^{(2)}}$. 
Then 
$X^{(1)}|_{\p M}=X^{(2)}|_{\p M}$ in $H^{1/2}(\p M, TM|_{\p M})$. 
\end{thm}

\begin{proof}
We shall first follow \cite{Brown_2001}. Let $x_0\in \p M$ and let $(x_1,\dots, x_n)$ be the boundary normal coordinates centered at $x_0$ so that in these coordinates, $x_0 =0$, the boundary $\p M$ is given by $\{x_n=0\}$, and $M^0$ is given by $\{x_n > 0\}$. 

Localizing near $x_0$ and passing to the boundary normal coordinates, we obtain a vector field $X=X^{(j)}\in (H^1\cap L^\infty)(\R^n_+, \R^n)$,  $j=1,2$, supported in a small neighborhood of $0$.  From  \cite[Section 1.1.3, p. 8]{Mazya_book} and  \cite{Brown_2001}, we know that $X$ has a representative such that for a.a. $x'\in B(0,r)=\{x'\in\R^{n-1}:|x'|<r\}$, $r>0$, the function 
\begin{equation}
\label{eq_app_H1}
x_n\mapsto X(x',x_n)\in \R^n\quad\text{is absolutely continuous}.
\end{equation}
Let $\mu>0$ be small and define the function, 
\[
X^*_\mu(x')=\sup_{0<x_n<\mu} |X(x',x_n)-X(x',0)|.
\]
Thus, for a.a. $x'\in B(0,r)$, we have $X^*_\mu(x')\to 0$ as $\mu\to 0^+$.  Using that $X\in L^\infty$ together with \eqref{eq_app_H1}, we see that 
$X^*_\mu(x')\in L^\infty(\R^{n-1})$ uniformly in $\mu$.  By the dominated convergence theorem, $X^*_\mu\to 0$ in $L^p(B(0,r))$ for $1\le p<\infty$.  Define the Hardy--Littlewood maximal operator on $\R^{n-1}$ by 
\[
M(f)(x')=\sup_{s>0} s^{1-n}\int_{|x'-y'|<s} |f(y')|dy', 
\]
and recall that  $M:L^p(\R^{n-1})\to L^p(\R^{n-1})$, $1<p<\infty$, see \cite{Stein_book}. 
As  $X^*_{\mu}(x')$ is decreasing as $\mu\to 0$, we get $0\le M(1_{B(0,r)} X^*_{\mu})$ is decreasing and $M(1_{B(0,r)} X^*_{\mu})\to 0$ in $L^p (\R^{n-1})$, $1<p<\infty$, as $\mu\to 0$. Thus, by Fatou's lemma, we obtain that 
\[
\int \liminf_{\mu \to 0} (M(1_{B(0,r)} X^*_{\mu})^2dx'\le \liminf_{\mu\to 0} \|M(1_{B(0,r)} X^*_{\mu})\|^2_{L^2}=0, 
\] 
and therefore, 
\begin{equation}
\label{eq_app_H2}
\lim_{\mu\to 0}M( 1_{B(0,r)}  X^*_\mu)(x')=0, \quad \text{for a.a.} \quad x'.
\end{equation} 
Furthermore, since the function $x'\mapsto X(x',0)$ is in $L^1_{\loc}(\R^{n-1})$, by the Lebesgue differentiation theorem  almost every point $x'$ is a Lebesgue point, i.e. 
\begin{equation}
\label{eq_app_H3}
\lim_{r\to 0^+}\frac{1}{r^{n-1}} \int_{|x'-y'|<r}|X(x',0)-X(y',0)|dy'=0. 
\end{equation}

Now we know from \cite[Section 1.1.3, p. 9]{Mazya_book}  that for a.a. $x'$, 
\[
X(x',x_n)=\int_0^{x_n} \frac{\p X}{\p t}(x',t)dt+f(x'), \quad f\in L^2_{\text{loc}}(\R^{n-1}),
\]
where $X$ is the representative  given in \eqref{eq_app_H1}. 
It follows that  
\[
X\in C(\R_{x_n},L^2_{\text{loc}}(\R^{n-1})). 
\]
 On the other hand, by Sobolev's trace theorem,   $H^1(\R^n)\subset C(\R_{x_n},H^{1/2}(\R^{n-1}))$, see \cite[p. 54]{Eskin_book}.  We conclude that 
 that the Sobolev trace of $X$ along $x_n=0$ agrees with $X(x',0)$ for almost all $x'$.

In order to proceed, similarly to \cite{Brown_2001}, we shall assume that $x'\in B(0,r)$ satisfies  \eqref{eq_app_H1}, \eqref{eq_app_H2} and \eqref{eq_app_H3} for both $j=1,2$.  Without loss of generality, we assume that $x'=0$.

Let us next show that  for $j=1,2$,
\begin{equation}
\label{eq_integral_0_app}
\langle ( \Lambda_{X^{(j)}} -\Lambda_0)f, \overline f \rangle_{H^{-1/2}(\p M)\times H^{1/2}(\p M)}=  \int_M X^{(j)}(u^{(j)})\overline{v}dV,
\end{equation}
 for all $u^{(j)}, v\in H^1(M^0)$ satisfying  
\begin{equation}
\label{eq_9_10}
\begin{aligned}
(-\Delta+X^{(j)})u^{(j)}&=0\quad \textrm{in}\quad \mathcal{D}'(M^0),\\
u^{(j)}|_{\p M}&=f,
\end{aligned}
\end{equation}
and 
\begin{equation}
\label{eq_9_10_v}
\begin{aligned}
-\Delta v&=0\quad \textrm{in}\quad \mathcal{D}'(M^0),\\
v|_{\p M}&=f,
\end{aligned}
\end{equation}
 and $f\in H^{1/2}(\p M)$. Indeed, by \eqref{eq_int_2_2} and \eqref{eq_int_2_2_Dirichlet_to_Neumann}, we get 
 \[
 \langle \Lambda_{X^{(j)}}f, \overline{f} \rangle_{H^{-1/2}(\p M)\times H^{1/2}(\p M)}= \int_M \langle \nabla u^{(j)},  \nabla \overline{v}\rangle dV +  \int_M X^{(j)}(u^{(j)})\overline{v}dV,
 \]
 and 
 \begin{align*}
 \langle \Lambda_{0}\overline{f}, f \rangle_{H^{-1/2}(\p M)\times H^{1/2}(\p M)}&= \int_M \langle \nabla\overline{v} ,  \nabla u^{(j)}\rangle dV\\
 &= \int_M \langle \nabla\overline{v} ,  \nabla v\rangle dV = \langle \Lambda_{0} f, \overline{f} \rangle_{H^{-1/2}(\p M)\times H^{1/2}(\p M)},
 \end{align*}
and therefore, taking the difference we obtain \eqref{eq_integral_0_app}. 
 
Now as $\Lambda_{X^{(1)}}=\Lambda_{X^{(2)}}$, we conclude from  \eqref{eq_integral_0_app} that 
\begin{equation}
\label{eq_integral_0_app_new}
\int_M X^{(1)}(u^{(1)})\overline{v}dV= \int_M X^{(2)}(u^{(2)})\overline{v}dV,
\end{equation}
 for all $u^{(j)}, v\in H^1(M^0)$ solutions to  \eqref{eq_9_10} and \eqref{eq_9_10_v}.

Our goal is to construct some special solutions to \eqref{eq_9_10} and \eqref{eq_9_10_v}, whose boundary values have an oscillatory behavior while becoming increasingly concentrated near $x_0=0$. When doing so, we shall assume, as we may, that 
\begin{equation}
\label{eq_Laplace_boundary-nc_2}
g^{\alpha \beta}(0)=\delta^{\alpha \beta}, \quad 1\le \alpha,\beta\le n-1,
\end{equation}
and therefore $T_0\p M=\R^{n-1}$, equipped with the Euclidean metric.  The unit tangent vector $\tau$ is then given by $\tau=(\tau',0)$ where $\tau'\in \R^{n-1}$, $|\tau'|=1$.   Associated to the tangent vector $\tau'$ is the covector $\xi'_\alpha=\sum_{\beta=1}^{n-1} g_{\alpha \beta}(0) \tau'_\beta=\tau'_\alpha\in T^*_{x_0}\p M$.
Let $\eta\in C^\infty_0(\R^n,\R)$ be a function such that $\supp(\eta)$ is in a small neighborhood of $0$, and 
\begin{equation}
\label{eq_int_eta_1}
\int_{\R^{n-1}}\eta(x',0)^2dx'=1.
\end{equation}
Following  \cite{Brown_2001}, \cite{Brown_Salo_2006}, in the boundary normal coordinates,  we set 
\begin{equation}
\label{eq_9_1}
v_0(x)=\eta\bigg(\frac{x}{\lambda^{1/2}}\bigg)e^{\frac{i}{\lambda}(\tau'\cdot x'+ ix_n)}, \quad 0<\lambda\ll 1,
\end{equation}
so that  $v_0\in C^\infty(M)$ with $\supp(v_0)$ in  $\mathcal{O}(\lambda^{1/2})$ neighborhood of $x_0=0$. Here $\tau'$ is viewed as a covector. 

Now we set
\[
f=v_0|_{\p M},
\]
and thus, $v=v_0+v_1$ solves  \eqref{eq_9_10_v} if $v_1\in H^1_0(M^0)$ is the unique solution to the following Dirichlet problem for the Laplacian,
\begin{equation}
\label{eq_9_Dirichlet_lapl}
\begin{aligned}
-\Delta v_1=&\Delta v_0, \quad\textrm{in}\quad \mathcal{D}'(M^0),\\
v_1|_{\p M}=&0.   
\end{aligned}
\end{equation}

Let $\delta(x)$ be the distance from $x\in M$ to the boundary of $M$.  Similarly  to \cite{Brown_Salo_2006}, we shall need the following estimates, established in \cite{Krup_Uhlmann_magnet_manifolds} in the case of Riemannian manifolds,
\begin{equation}
\label{eq_9_3} 
\|v_0\|_{L^2(M)}\le \mathcal{O}(\lambda^{\frac{n-1}{4}+\frac{1}{2}}),
\end{equation}
\begin{equation}
\label{eq_9_4}
\|v_1\|_{L^2(M)}\le \mathcal{O}(\lambda^{\frac{n-1}{4}+\frac{1}{2}}),
\end{equation}
\begin{equation}
\label{eq_9_15}
\|d v_1\|_{L^2(M)}\le \mathcal{O}(\lambda^{\frac{n-1}{4}}),
\end{equation}
\begin{equation}
\label{eq_Brown_Salo_2_18}
\|\delta dv_0\|_{L^2(M)}\le \mathcal{O}(\lambda^{\frac{n-1}{4}+\frac{1}{2}}),
\end{equation}
We shall also need Hardy's inequality, 
\begin{equation}
\label{eq_Hardy}
\int_{M}|\psi (x)/\delta(x)|^2 dV\le C\int_{M}|d \psi(x)|^2dV,
\end{equation}
where $\psi\in H^1_0(M^0)$, see \cite{Davies_2000}.

Now $u^{(j)}=v_0+w^{(j)}$ solves  \eqref{eq_9_10} if $w^{(j)}\in H^1_0(M^0)$ is the unique solution to 
\begin{equation}
\label{eq_w_j_Lax}
\begin{aligned}
(-\Delta+X^{(j)})w^{(j)}&=-(\Delta+X^{(j)})v_0\quad \textrm{in}\quad \mathcal{D}'(M^0),\\
w^{(j)}|_{\p M}&=0.
\end{aligned}
\end{equation}
Furthermore, combining the Lax--Milgram lemma with the uniqueness of solutions to \eqref{eq_w_j_Lax}, we get
\begin{equation}
\label{eq_9_10_u_1_es}
\|w^{(j)}\|_{H^{1}(M^0)}\le C\| (\Delta+X^{(j)})v_0\|_{H^{-1}(M^0)}.
\end{equation}
In order to estimate the right hand side of \eqref{eq_9_10_u_1_es}, we  recall the following bound from \cite[Appendix]{Krup_Uhlmann_magnet_manifolds}, 
\begin{equation}
\label{eq_9_10_u_1_es_2}
\|\Delta v_0\|_{H^{-1}(M^0)}\le \mathcal{O}(\lambda^{\frac{n-1}{4}}).
\end{equation}
Letting $\varphi\in C^\infty_0(M^0)$ and using \eqref{eq_Hardy}, \eqref{eq_Brown_Salo_2_18}, we get
\[
|\langle  X^{(j)}v_0, \varphi\rangle_{M^0} |\le \mathcal{O}(1)\| \delta dv_0\|_{L^2(M)}\|\varphi\|_{H^{1}(M^0)}\le 
\mathcal{O}(\lambda^{\frac{n-1}{4}+\frac{1}{2}}),
\]
showing that 
\begin{equation}
\label{eq_9_10_u_1_es_3}
\| X^{(j)}v_0\|_{H^{-1}(M^0)}\le \mathcal{O}(\lambda^{\frac{n-1}{4}+\frac{1}{2}}).
\end{equation}
It follows from \eqref{eq_9_10_u_1_es}, \eqref{eq_9_10_u_1_es_2} and \eqref{eq_9_10_u_1_es_3} that 
\begin{equation}
\label{eq_9_10_u_1_es_4}
\|w^{(j)}\|_{H^{1}(M^0)}\le \mathcal{O}(\lambda^{\frac{n-1}{4}}).
\end{equation}

The next step is to substitute the solutions $u^{(j)}$ and $v$ of  \eqref{eq_9_10} and \eqref{eq_9_10_v} into the identity \eqref{eq_integral_0_app_new}, multiply it by $\lambda^{-\frac{(n-1)}{2}}$ and compute the limit as $\lambda\to 0$.   To this end, we write $X=X^{(j)}$,  $u=u^{(j)}$, $w=w^{(j)}$,  $j=1,2$, as before, and let
 \begin{equation}
\label{eq_def_I}
I:= \lambda^{-\frac{(n-1)}{2}}  \int_M X(u)\overline{v}  dV=I_1+I_2+I_3,
\end{equation}
where 
\begin{align*}
I_1&= \lambda^{-\frac{(n-1)}{2}}   \int_M X(v_0)\overline{v_0} dV,\\
I_2&= \lambda^{-\frac{(n-1)}{2}}  \int_M X(v_0)\overline{v_1}dV,\\
I_3&= \lambda^{-\frac{(n-1)}{2}}  \int_M X(w)(\overline{v_0}+\overline{v_1})  dV.
\end{align*}

Let us compute $\lim_{\lambda\to 0}I_1$.   To that end, writing $X=X_j\p_{x_j} $, we  first get
\begin{equation}
\label{eq_v_0_grad}
X( v_0)= e^{\frac{i}{\lambda}(\tau'\cdot x'+i x_n)}  \bigg[ \lambda^{-1/2} X(x)\cdot (\nabla \eta)\bigg(\frac{x}{\lambda^{\frac{1}{2}}}\bigg) +  i\lambda^{-1} X(x)\cdot (\tau', i)  \eta\bigg(\frac{x}{\lambda^{\frac{1}{2}}}\bigg) \bigg],
\end{equation}
and 
\[
X( v_0)\overline{v}_0= e^{-\frac{2x_n}{\lambda}}   \bigg[ \lambda^{-1/2} X(x)\cdot  (\nabla \eta)\bigg(\frac{x}{\lambda^{\frac{1}{2}}}\bigg) \eta\bigg(\frac{x}{\lambda^{\frac{1}{2}}} \bigg)+  i\lambda^{-1} X(x)\cdot (\tau',i)  \bigg|\eta\bigg(\frac{x}{\lambda^{\frac{1}{2}}}\bigg)\bigg|^2 \bigg]
\]
Since
\[
\bigg|\lambda^{-\frac{(n-1)}{2}}   \int_{M}   e^{-\frac{2x_n}{\lambda}}    \lambda^{-1/2}X(x) \cdot (\nabla \eta)\bigg(\frac{x}{\lambda^{\frac{1}{2}}}\bigg) \eta\bigg(\frac{x}{\lambda^{\frac{1}{2}}} \bigg) dV\bigg|=\mathcal{O}(\lambda^{1/2}),
\]
as $\lambda\to 0$, we get
\begin{equation}
\label{eq_app_100_-2}
\begin{aligned}
\lim_{\lambda\to 0} I_1&= \lim_{\lambda\to 0} i \lambda^{-\frac{(n+1)}{2}} \int_{\R^n,x_n> 0} (\tau',i) \cdot X(x)  \bigg|\eta\bigg(\frac{x}{\lambda^{\frac{1}{2}}}\bigg)\bigg|^2 e^{-\frac{2x_n}{\lambda}} |g(x)|^{1/2}dx\\
&= \lim_{\lambda\to 0}(I_{1,1}+I_{1,2}+I_{1,3}),
\end{aligned}
\end{equation}
where 
\begin{align*}
I_{1,1}&=i \lambda^{-\frac{(n+1)}{2}}(\tau',i) \cdot  X(0)  \int_{\R^n,  x_n> 0}  \bigg|\eta\bigg(\frac{x}{\lambda^{\frac{1}{2}}}\bigg)\bigg|^2 e^{-\frac{2x_n}{\lambda}} |g(x)|^{1/2}dx,\\
I_{1,2}&=i \lambda^{-\frac{(n+1)}{2}} \int_{\R^n, 0< x_n< \mu}  (\tau',i) \cdot (  X(x',x_n) -  X(0,0))   \bigg|\eta\bigg(\frac{x}{\lambda^{\frac{1}{2}}}\bigg)\bigg|^2 e^{-\frac{2x_n}{\lambda}} |g(x)|^{1/2}dx, \\
I_{1,3}&=i \lambda^{-\frac{(n+1)}{2}} \int_{\R^n,  x_n> \mu} (\tau',i) \cdot ( X(x',x_n) - X(0,0))   \bigg|\eta\bigg(\frac{x}{\lambda^{\frac{1}{2}}}\bigg)\bigg|^2 e^{-\frac{2x_n}{\lambda}} |g(x)|^{1/2}dx,
\end{align*}
for some $\mu>0$ to be chosen.  
Making the change of variables $x'=\lambda^{1/2}y'$, $x_n=\lambda y_n$, and using \eqref{eq_Laplace_boundary-nc_2}, \eqref{eq_int_eta_1}, we obtain that 
\begin{equation}
\label{eq_app_100_-1}
\begin{aligned}
\lim_{\lambda\to 0} I_{1,1}&=i (\tau',i) \cdot X(0) \lim_{\lambda\to 0} \int_{\R^n,  y_n> 0}  |\eta(y',\lambda^{1/2}y_n)|^2  e^{- 2y_n} |g(\lambda^{1/2}y', \lambda y_n)|^{1/2}dy\\
&=\frac{i}{2}(\tau',i) \cdot  X(0).
\end{aligned}
\end{equation}
Using that $ X\in L^\infty$, we also get 
\begin{equation}
\label{eq_app_100_0}
|I_{1,3}|\le \mathcal{O}(\lambda^{-\frac{(n+1)}{2}}\lambda^{\frac{n-1}{2}} e^{-\frac{2\mu}{\lambda}})=\mathcal{O}_\mu(\lambda^{\infty}).
\end{equation}
We have furthermore 
\begin{equation}
\label{eq_app_100_1}
\begin{aligned}
|I_{1,2}|&\le \mathcal{ O}(\lambda^{-\frac{(n+1)}{2}}) \int_{\R^n, 0< x_n< \mu}  |  X(x',x_n) -  X(0,0)|   \bigg|\eta\bigg(\frac{x}{\lambda^{\frac{1}{2}}}\bigg)\bigg|^2 e^{-\frac{2x_n}{\lambda}} |g(x)|^{1/2}dx\\
&\le  \mathcal{ O}(\lambda^{-\frac{(n+1)}{2}}) (I_{1,2,1}+ I_{1,2,2}), 
\end{aligned}
\end{equation}
where 
\begin{align*}
I_{1,2,1}&=\int_{\R^n, 0< x_n< \mu}  | X(x',x_n) -  X(x',0)|   \bigg|\eta\bigg(\frac{x}{\lambda^{\frac{1}{2}}}\bigg)\bigg|^2 e^{-\frac{2x_n}{\lambda}} |g(x)|^{1/2}dx,\\
I_{1,2,2}&=\int_{\R^n, 0< x_n< \mu}  |  X(x',0) -  X(0,0)|   \bigg|\eta\bigg(\frac{x}{\lambda^{\frac{1}{2}}}\bigg)\bigg|^2 e^{-\frac{2x_n}{\lambda}} |g(x)|^{1/2}dx.
\end{align*}
Using that $0\in \R^{n-1}$ is a Lebesgue point for the function $x'\mapsto X(x',0)$, by \eqref{eq_app_H3}, we get 
\begin{equation}
\label{eq_app_100_2}
\begin{aligned}
|I_{1,2,2}|&\le \mathcal{O} (\lambda)  \int_{|x'|\le c\lambda^{1/2}}  |  X(x',0) -  X(0,0)|  dx'\\
&\le \mathcal{O} (\lambda\lambda^{\frac{n-1}{2}})\frac{1}{\lambda^{\frac{n-1}{2}}} \int_{|x'|\le c\lambda^{1/2}}  |  X(x',0) -  X(0,0)|  dx'=
o(\lambda^{\frac{n+1}{2}}).
\end{aligned}
\end{equation}
By  \eqref{eq_app_H2}, we obtain that  
\begin{equation}
\label{eq_app_100_3}
\begin{aligned}
|I_{1,2,1}|&\le \mathcal{O}(\lambda) \int_{|x'|\le c\lambda^{1/2}} \sup_{0<x_n<\mu} |  X(x',x_n) -  X(x',0)|  dx' \\
&=\mathcal{O}(\lambda) \int_{|x'|\le c\lambda^{1/2}}  X^*_\mu(x')dx'\le  \mathcal{O}(\lambda \lambda^{\frac{n-1}{2}})\frac{1}{\lambda^{\frac{n-1}{2}}} \int_{|x'|\le c\lambda^{1/2}}  X^*_\mu(x')dx'\\
&\le \mathcal{O}(\lambda^{\frac{n+1}{2}}) M(X^*_\mu) (0)=\mathcal{O}(\lambda^{\frac{n+1}{2}})o(1), 
\end{aligned}
\end{equation}
as $\mu\to 0^+$. It follows from \eqref{eq_app_100_0}, \eqref{eq_app_100_1}, \eqref{eq_app_100_2} and \eqref{eq_app_100_3} that  
$\lim_{\lambda\to 0}(I_{1,2}+I_{1,3})=0$. Combining this with \eqref{eq_app_100_-2} and \eqref{eq_app_100_-1}, we get 
\begin{equation}
\label{eq_app_100_4}
\lim_{\lambda\to 0}I_1=\frac{i}{2}(\tau',i) \cdot  X(0). 
\end{equation}

Recalling that  $v_1\in H^1_0(M^0)$ and using \eqref{eq_Hardy}, \eqref{eq_Brown_Salo_2_18} and \eqref{eq_9_15}, we obtain that 
\begin{equation}
\label{eq_app_100_5}
\begin{aligned}
|I_2|&\le \mathcal{O}(\lambda^{-\frac{(n-1)}{2}}) \|X\|_{L^\infty(M)}\|\delta d v_0\|_{L^2(M)} \| v_1/\delta\|_{L^2(M)}\\
&\le \mathcal{O}(\lambda^{-\frac{(n-1)}{2}})\mathcal{O}(\lambda^{\frac{n-1}{4}+\frac{1}{2}})  \| dv_1\|_{L^2(M)}=\mathcal{O}(\lambda^{\frac{1}{2}}). 
\end{aligned}
\end{equation}

By    \eqref{eq_9_3}, \eqref{eq_9_4}, and \eqref{eq_9_10_u_1_es_4},  we get 
\begin{equation}
\label{eq_app_100_6}
|I_3|\le  \mathcal{O}(\lambda^{-\frac{(n-1)}{2}}) \|X\|_{L^\infty(M)}\| d w\|_{L^2(M)} \| v_0+v_1\|_{L^2(M)}=\mathcal{O}(\lambda^{\frac{1}{2}}). 
\end{equation}

Now it follows from \eqref{eq_def_I}, \eqref{eq_app_100_4}, \eqref{eq_app_100_5}, and  \eqref{eq_app_100_6} that 
\[
\lim_{\lambda\to 0}I=\frac{i}{2}(\tau',i) \cdot  X(0),
\]
and therefore, \eqref{eq_integral_0_app_new} implies that 
\[
(\tau',i) \cdot  X^{(1)}(0)=(\tau',i) \cdot  X^{(2)}(0),
\]
for all $\tau'\in \R^{n-1}$. This completes the proof of Theorem \ref{prop_boundary_rec}. 
\end{proof}

A simplified version of the discussion above gives also the following result. 
\begin{cor} 
\label{cor_boundary_rec}
Let $(M,g)$ be a smooth compact Riemannian manifold of dimension $n\ge 3$ with smooth boundary $\p M$, and  let $X^{(1)}, X^{(2)}\in C(M, TM)$.  Assume that $\Lambda_{ X^{(1)}}=\Lambda_{ X^{(2)}}$. 
Then 
$X^{(1)}|_{\p M}=X^{(2)}|_{\p M}$. 
\end{cor}

\section{Approximation estimates}

\label{sec_approx_est}

The purpose of this appendix is to collect some approximation results which are used in the main part of the paper. The estimates are well known and are given here for the convenience of the reader, see \cite{Krup_Uhlmann_calderon}, \cite{Zhang}.

Let $\Psi_\tau(x)=\tau^{-n}\Psi(x/\tau)$, $\tau>0$, be the usual mollifier with $\Psi\in C^\infty_0(\R^n)$, $0\le \Psi\le 1$, and $\int \Psi dx=1$. 

\begin{prop}
\label{prop_app_R_n}
Let $f\in H^s(\R^n)$ with some  $0\le s\le 1$. Then $f_\tau=f*\Psi_\tau\in (H^s\cap C^\infty)(\R^n)$ satisfies the following estimates
\begin{equation}
\label{eq_approxim_1}
\|f-f_\tau\|_{L^2(\R^n)}=o(\tau^s),\quad \tau\to 0,
\end{equation}
\begin{equation}
\label{eq_approxim_2}
\begin{aligned}
\|f_\tau\|_{L^2(\R^n)}=&\mathcal{O}(1), \quad \|\nabla f_\tau\|_{L^2(\R^n)}=\begin{cases} o(\tau^{s-1}), & 0\le s<1,\\
\mathcal{O}(1), & s=1,
\end{cases} \\
 &\|\p^\alpha f_\tau\|_{L^2(\R^n)}= o(\tau^{s-|\alpha|}),  \quad |\alpha|\ge 2, 
\end{aligned}
\end{equation}
as $\tau\to 0$. Furthermore, if $f\in L^\infty(\R)$, we have
\begin{equation}
\label{eq_approxim_3}
\|f_\tau\|_{L^\infty(\R^n)}=\mathcal{O}(1), \quad \|\p^\alpha f_\tau\|_{L^\infty(\R^n)}= \mathcal{O}(\tau^{-|\alpha|}), \quad  |\alpha|\ge 1, 
\end{equation}
as $\tau\to 0$. 
\end{prop}

\begin{proof}

Let us start by proving \eqref{eq_approxim_1} when $0\le s<1$. First we have  $\hat \Psi_\tau(\xi)=\hat\Psi(\tau \xi)$,
where $\hat{\Psi}$ stands for the Fourier transform of $\Psi$, and therefore, $\hat f_\tau(\xi)=\hat f(\xi)\hat\Psi(\tau \xi)$. 
We get 
\begin{equation}
\label{eq_approxim_4}
\begin{aligned}
\tau^{-2s}\|f-f_\tau\|_{L^2(\R^n)}^2=\tau^{-2s}(2\pi)^{-n}\int_{\R^n} |1-\hat\Psi(\tau\xi)|^2|\hat f(\xi)|^2d\xi\\
=\int_{\R^n} g(\tau \xi) |\xi|^{2s}|\hat f(\xi)|^2d\xi,
\end{aligned}
\end{equation}
where
\[
g(\eta)=(2\pi)^{-n}\frac{|1-\hat \Psi(\eta)|^2}{|\eta|^{2s}}.
\]
As $\hat \Psi(0)=1$,  we have $\hat \Psi(\eta)=1+\mathcal{O}(|\eta|)$, and therefore, when $0\le s<1$,  we get $g(0)=0$.  Furthermore, as $\Psi\in C^\infty_0(\R^n)$, we see that $g$ is continuous and bounded. 
By Lebesgue's dominated convergence theorem applied to \eqref{eq_approxim_4}, using that   $f\in H^s(\R^n)$, we obtain that $\tau^{-2s}\|f-f_\tau\|_{L^2(\R^n)}^2\to 0$ as $\tau\to 0$, showing \eqref{eq_approxim_1} when $0\le s<1$. 

When proving  \eqref{eq_approxim_1} in the case $s=1$, we assume furthermore that $\Psi$ is a radial function. Then $\nabla \hat \Psi(0)=-i\int_{\R^n} x\Psi(x)dx=0$, and therefore, $\hat \Psi(\eta)=1+\mathcal{O}(|\eta|^2)$. The proof of \eqref{eq_approxim_1} in the case $s=1$ thus proceeds similarly to the one when $0\le s<1$.  

Let us now show \eqref{eq_approxim_2}. Here we do not need the assumption that $\Psi$ is radial. First by Young's inequality, we get
\[
\|f_\tau\|_{L^2(\R^n)}\le \|f\|_{L^2(\R^n)}\| \Psi_\tau\|_{L^1(\R^n)}=\|f\|_{L^2(\R^n)},
\] 
showing the first bound in \eqref{eq_approxim_2}.  For $|\alpha|\ge 1$,   we have
\begin{equation}
\label{eq_approxim_5}
\begin{aligned}
\tau^{2(|\alpha|-s)}\|\p^\alpha f_\tau\|_{L^2(\R^n)}^2 \le (2\pi)^{-n} \tau^{2(|\alpha|-s)} \int_{\R^n} |\xi|^{2|\alpha|} |\hat f(\xi)|^2|\hat \Psi(\tau\xi)|^2d\xi\\
=\int_{\R^n} \tilde g(\tau \xi) |\xi|^{2s}|\hat f(\xi)|^2d\xi,
\end{aligned}
\end{equation}
where
\[
\tilde g(\eta)=(2\pi)^{-n} |\eta|^{2(|\alpha|-s)}|\hat \Psi(\eta)|^2
\]
is continuous and bounded. 
Now when $|\alpha|\ge 2$, $0\le s\le 1$, or $|\alpha|=1$, $0\le s<1$, we have that $\tilde g(0)=0$, and thus, by Lebesgue's dominated convergence theorem applied to \eqref{eq_approxim_5}, we conclude that  $\tau^{2(|\alpha|-s)}\|\p^\alpha f_\tau\|_{L^2(\R^n)}^2 \to 0$ as $\tau\to 0$. Thus, \eqref{eq_approxim_2} is established in all cases except for $|\alpha|=s=1$.  In the latter case \eqref{eq_approxim_2} follows from \eqref{eq_approxim_5}.  The estimates \eqref{eq_approxim_3} are immediate. 

\end{proof}

We shall also  need the following well known  approximation estimates for H\"older spaces $C^{0,\gamma}(\R^n)$, see \cite[Lemma 3.1]{Pohjola}. Here 
\[
C^{0,\gamma}(\R^n)=\bigg\{f\in C(\R^n)\cap L^\infty(\R^n): \sup_{x\ne y}\frac{|f(x)-f(y)|}{|x-y|^{\gamma}}<\infty\bigg\}, \ 0<\gamma\le 1. 
\]
\begin{prop}
\label{prop_app_Holder_R_n}
Let $f\in C^{0,\gamma}(\R^n)$ with some $0<\gamma\le 1$. Then $f_\tau=f*\Psi_\tau\in C^\infty(\R^n)$  satisfies the estimates
\begin{align*}
\| f-f_\tau\|_{L^\infty(\R^n)}=\mathcal{O}(\tau^{\gamma}), \quad \|\p ^\alpha f_\tau\|_{L^\infty(\R^n)}=\mathcal{O}(\tau^{\gamma-|\alpha|}), \quad |\alpha|\ge 1,
\end{align*}
as $\tau\to 0$. 
\end{prop}

\end{appendix}

\section*{Acknowledgements}
The research of K.K. is partially supported by the National Science Foundation (DMS 1500703). The research of G.U. is partially supported by NSF, a Si-Yuan Professorship of HKUST and FiDiPro Professorship of the Academy of Finland.

\end{document}